\documentclass[dvipsnames,a4paper,twoside]{article}
\usepackage{amssymb,amsmath}
\usepackage{hyperref}
\usepackage{graphicx}
\usepackage{xspace}
\usepackage[ulem=normalem,draft]{changes}
\usepackage[textsize=small]{todonotes}\reversemarginpar

\newcommand{\Pb}{\mbox{\rm (P)}\xspace}
\newcommand{\Pbt}{\mbox{\rm (P$_\tau$)}\xspace}
\newcommand{\Pbtc}{\mbox{\rm (P$_\tau^c$)}\xspace}

\newcommand{\dx}{\, \textup{d}x}
\newcommand{\dt}{\, \textup{d}t}
\newcommand{\dtau}{\, \textup{d}\tau}
\newcommand{\dus}{\, \textup{d}u(s)}
\newcommand{\dbus}{\, \textup{d}\bar u(s)}

\newcommand{\duas}{\, \textup{d}|u|(s)}
\newcommand{\dbuas}{\, \textup{d}|\bar u|(s)}

\newcommand{\CT}{{C[0,T]}}
\newcommand{\MT}{{\mathcal{M}[0,T]}}
\newcommand{\U}{{\mathcal{U_\tau}}}

\newcommand{\rli}{{L^\infty(Q)}}

\usepackage{amsthm}
\newtheorem{theorem}{Theorem}[section]
\newtheorem{lemma}[theorem]{Lemma}
\newtheorem{definition}[theorem]{Definition}
\newtheorem{proposition}[theorem]{Proposition}
\newtheorem{corollary}[theorem]{Corollary}

\newtheorem{remark}[theorem]{{\it Remark }\rm }
\newtheorem{example}[theorem]{{\it Example}\rm }

\title{Measure Control of a Semilinear Parabolic Equation with a Nonlocal Time Delay
\thanks{The first two authors were partially supported by the Spanish Ministerio de
Econom\'{\i}a y Competitividad under projects MTM2014-57531-P and MTM2017-83185-P.
The third author was supported by the collaborative research center
SFB 910, TU Berlin, project  B6.}}

\author{Eduardo Casas\thanks{Departmento de Matem\'{a}tica Aplicada y Ciencias de la Computaci\'{o}n, E.T.S.I. Industriales y de Telecomunicaci\'on, Universidad de Cantabria, 39005 Santander, Spain, {\tt eduardo.casas@unican.es}.}
\and Mariano Mateos\thanks{Departamento de Matem\'{a}ticas, Campus de Gij\'on, Universidad de Oviedo, 33203, Gij\'on, Spain, {\tt mmateos@uniovi.es}.}
\and Fredi Tr\"oltzsch\thanks{Institut f\"ur Mathematik, Technische Universit\"at Berlin, D-10623 Berlin,
Germany,  {\tt troeltzsch@math.tu-berlin.de.}}}

\begin{document}
\maketitle

\begin{abstract}
  We study a control problem governed by a semilinear parabolic equation. The control is a measure that acts as
  the kernel of a possibly nonlocal time delay term and the functional includes a non-differentiable term with
  the measure-norm of the control. Existence, uniqueness and regularity of the solution of the state equation, as
  well as differentiability properties of the control-to-state operator are obtained. Next, we provide first order
  optimality conditions for local solutions. Finally, the control space is suitably discretized and we prove convergence of
  the solutions of the discrete problems to the solutions of the original problem. Several numerical examples are
  included to illustrate the theoretical results.
\end{abstract}



\section{Introduction}
\label{S1}
\setcounter{equation}{0}

We consider optimal control problems for the parabolic equation
\begin{equation}
\left\{\begin{array}{rcll}\displaystyle\frac{\partial y}{\partial t} - \Delta y + R(y) &=
&\displaystyle \int_0^T y(x,t-s)\dus \ &\text{ in } Q = \Omega \times (0,T),\vspace{2mm}\\
\displaystyle \partial_ny &=& 0 \ &\text{ on } \Sigma = \Gamma \times (0,T),\\[1ex]
y(x,t) &=& y_0(x,t) \ &\text{ in }Q_- =  \Omega \times [-T,0], \end{array}\right.
\label{E1.1}
\end{equation}
where the Borel measure $u \in \MT$ is taken as control. Depending on the particular choice of this measure
and on the  form of the nonlinearity $R$, different mathematical models of interest for theoretical physics
are covered by this equation. Thanks to its generality, this equation includes the control of  time delays in
parabolic equations, the control of multiple time delays, and also the optimization of standard feedback operators of
Pyragas type. Associated examples will be explained below.

Our paper extends the optimization of nonlocal Pyragas type feedback operators that was investigated in \cite{nestler_schoell_troeltzsch2015}. The main novelty of our paper is the use of measures instead of functions. This is much more general and
leads to  new, partially delicate and interesting questions of analysis.
The partial differential equation above includes three main difficulties: First, the equation is of semilinear type.
Main ideas for the associated analysis were prepared 
{in} \cite{casas_ryll_troeltzsch2014} and we are able to
proceed similarly, at least partially. Second, the equation contains some kind of time delays.
Finally, the integral operator includes the measure $u$ that complicates the analysis.

The optimal control theory of ordinary or partial differential equations with time-delay has a very long history. Numerous papers were contributed to this field. We mention exemplarily the papers \cite{Banks_Burns1978,Banks_Burns_Cliff1981,Colonius_Hinrichsen1976,Kappel_Kunisch1981,Kunisch1981},
that have some relation to distributed parameter systems,  or the surveys \cite{Banks_Manitius1974,Richard2003}. More recent contributions are e.g. \cite{Jeong_Hwang2015, Mordukhovich_Wang2009,
Mordukhovich_Wang2010}.

However, to our best knowledge, the optimal
control of parabolic equations with nonlocal time delay was only investigated in
\cite{nestler_schoell_troeltzsch2015}. The case of measures as controls  is new for
this type of equations.
{However, we mention \cite{Ahmed2003}, where a measure-valued control function is considered in a delay equation.}

Moreover, the control is not taken as a right-hand side. Here, it plays the role of a kernel
in an integral operator; this is another difficulty. We should mention that the use of kernels as control functions
is not new. For instance, memory kernels were taken as ''controls'' in  identification problems in
\cite{unger_vonwolfersdorf1995} and \cite{vonwolfersdorf1993}.

The equation above generalizes different models of Pyragas type feedback that are very popular in theoretical
physics.
We mention the seminal paper by Pyragas \cite{pyragas1992}, where the feedback of the form \eqref{E1.2}  was
introduced to stabilize periodic orbits; see also \cite{pyragas2006}. We also refer to
\cite{schoell_schuster2008, siebert_alonso_baer_schoell14, siebert_schoell14},
where  nonlocal Pyragas  feedback operators of the type \eqref{E1.5} are discussed for different kernels $u$.

Let us also mention \cite{kyrychko_blyuss_schoell11}, where the implemention of nonlocal
feedback controllers is investigated.
In particular,  these equations have applications in Laser technology;
{we refer to associated contributions in \cite{schoell_schuster2008}.}

Let us also mention a few examples for the equation \eqref{E1.1}. In the following $\kappa$ is a real parameter.

\begin{example}[Pyragas feedback control]
If $\tau \in (0,T)$ is a fixed time, and $\delta_\tau$ and $\delta_0$ denote the Dirac measures concentrated at $\tau$ and $0$, respectively, then the equation
\begin{equation} \label{E1.2}
 \frac{\partial y}{\partial t} - \Delta y + R(y) =\kappa \left(y(x,t-\tau) - y(x,t)\right)
\end{equation}
is obtained as particular case of \eqref{E1.1} with $u = \kappa \, (\delta_\tau - \delta_0)$. Equations of this type with fixed time delay $\tau$ are known in the context
of the so-called  Pyragas type feedback control, \cite{pyragas1992,pyragas2006,schoell_schuster2008}.
\end{example}
\begin{example}[Pyragas feedback with multiple time delays]
A more general version of \eqref{E1.2} with multiple fixed time delays is
{generated} by
\begin{equation} \label{E1.3}
 u = \kappa\Big(\sum_{i=1}^m u_i \delta_{\tau_i} - \delta_0\Big)
\end{equation}
with fixed time delays $0 < \tau_1 < \ldots < \tau_m < T$. 
{Then the} equation
\begin{equation} \label{E1.4}
\frac{\partial y}{\partial t} - \Delta y + R(y) =\kappa \left( \sum_{i=1}^m u_i \, y(x,t-\tau_i) - y(x,t)\right)
\end{equation}
is obtained. Here, the control $(u_1, \ldots, u_m) \in \mathbb{R}^m$  is a vector of controllable weights.
\end{example}
\begin{example}[Nonlocal Pyragas feedback control]
Finally, if the Lebesgue decomposition of $u$ is $u=\kappa( u_r-\delta_0)$, where $u_r$ is absolutely continuous w.r.t. the Lebesgue measure in $[0,T]$ and its Radon-Nikodym derivative is $g \in L^1(0,T)$,  then equation \eqref{E1.1} takes the form
\begin{equation} \label{E1.5}
\frac{\partial y}{\partial t} - \Delta y + R(y) =\kappa \left(\int_0^T g(s)y(x,t-s)\, ds - y(x,t)\right)
\end{equation}
that is used in nonlocal Pyragas type feedback. Here, the control is the  integrable
kernel $g$ of the integral operator of the partial differential equation.
\end{example}

\section{State Equation}
\label{S2}
\setcounter{equation}{0}

Throughout this paper $\MT$ will denote the space of real and regular Borel measures in $[0,T]$. According to the Riesz representation theorem, $\MT$ is the dual of the space of continuous functions in $[0,T]$: $\MT = \CT^*$, and $\MT$ is a Banach space endowed with the norm
\[
\|u\|_\MT = |u|([0,T]) = \sup\left\{\int_0^T\phi(s)\dus : \phi \in \CT \text{ and } \|\phi\|_\CT \le 1\right\}
\]
for any real numbers $a \le b$.
Here, $|u|$ denotes the total variation measure of $u$; see \cite[pp.~130--133]{Rudin70}. The above integrals are considered in the closed interval $[0,T]$. Notice that $u(\{0\})$ and $u(\{T\})$ could be nonzero. This notational convention will be maintained in the sequel. Thus, we distinguish
\[
\int_a^b\phi(s)\dus = \int_{[a,b]}\phi(s)\dus,\ \int_{[a,b)}\phi(s)\dus,\ \int_{(a,b]}\phi(s)\dus,\ \int_{(a,b)}\phi(s)\dus,
\]
for real numbers $a \le b$.

We recall our general state equation \eqref{E1.1},
\[
\left\{\begin{array}{rcll}\displaystyle\frac{\partial y}{\partial t} - \Delta y + R(y) &=&\displaystyle \int_0^T y(x,t-s)\dus \ &\text{ in } Q,\vspace{2mm}\\ \displaystyle \partial_ny &=& 0 \ &\text{ on } \Sigma,\\[1ex] y(x,t) &=& y_0(x,t) \ &\text{ in }Q_-. \end{array}\right.
\]
In this setting, $\Omega \subset \mathbb{R}^n$, $1 \le n \le 3$, is a bounded Lipschitz domain with boundary $\Gamma$ and,
(as introduced above) $Q_- = \Omega \times [-T,0]$.
 The nonlinearity $R: \mathbb{R} \longrightarrow \mathbb{R}$ is a function of class $C^1$ such that
\begin{equation}
\exists C_R \in \mathbb{R} : R'(y) \ge C_R \ \ \  \forall y \in \mathbb{R}.
\label{E2.1}
\end{equation}
The initial datum $y_0$ is taken from $C(\bar Q_-)$, while $u \in \MT$ is the control.

{
In particular, third order polynomials of the form
\[
R(y) = \rho (y - y_1)(y - y_2)(y - y_3)
\]
with $\rho > 0$ and $y_1 < y_2 < y_3$ satisfy this assumption. $R$ has the meaning of a reaction term. The numbers $y_i, \, i=1,2,3$, are the  fixed points of the reaction; $y_1$ and $y_3$ are the stable ones, while $y_2$ is unstable. Such functions $R$ play a role in bistable reactions of physical chemistry; see \cite{Lober2014}.

{The assumptions are also fulfilled by higher order polynomials of odd order
\[
R(y) = \sum_{i=0}^k a_i \, y^i,
\]
with real numbers $a_i$, $i = 0,\ldots,k$, and $k = 2 \ell + 1, \, \ell \in \mathbb{N}\cup \{0\}$, if $a_k$ is positive.
Here, the derivative $R'$ is an even order polynomial that satisfies condition \eqref{E2.1}.

\begin{remark}The theory of our paper can be extended to more general functions $R: \Omega \times \mathbb{R} \to \mathbb{R}$, that obey the following assumptions:
\begin{itemize}
\item $R$ is a Carath\'eodory function of class $C^1$ with respect to the second variable.
\item There exists some $p > n/2$ such that $R(\cdot,0) \in L^p(\Omega)$.
\item For all $M > 0$ there exists a constant $C_M > 0$ such that
\[
\Big|\frac{\partial R}{\partial y}(x,y)\Big| \le C_M\ \mbox { for a.a.}\ x \in \Omega \mbox{ and } \forall y \in \mathbb{R}  \mbox{ with } |y| \le M.
\]
\item The function $y \mapsto \frac{\partial R}{\partial y}(x,y)$ is  bounded from below, i.e.
\[
\frac{\partial R}{\partial y}(x,y) \ge C_R \quad \mbox{ for a.a. } x \in \Omega \mbox{ and all } y \in \mathbb{R}.
\]
\end{itemize}
The theory remains also true, if - in addition to these general assumptions on $R$ -
the Laplace operator $-\Delta$ is replaced by another uniformly elliptic differential operator $A$ with $L^\infty$ coefficients in the main part of the operator. Lipschitz regularity of these coefficients is required only in the second part of Theorem \ref{T2.2}.

However, to keep the presentation simple, we concentrate on the case $A = -\Delta$, and a function $R: \mathbb{R} \to \mathbb{R}$ of class $C^1$ and satisfying condition \eqref{E2.1}.
\label{R2.1}
\end{remark}

In the sequel, we will denote $Y = L^2(0,T;H^1(\Omega)) \cap C(\bar Q)$. Endowed with the norm
\[
\|y\|_Y = \|y\|_{L^2(0,T;H^1(\Omega))} + \|y\|_{C(\bar Q)}
\]
$Y$ is a Banach space.

We begin our analysis with the well-posedness of the state equation \eqref{E1.1} that is a differential equation with time delay. Ordinary differential delay equations are well understood, we  refer exemplarily to the expositions \cite{Bellman1954,Hale_Verduyn1993,Erneux2009}. For parabolic partial differential equations,
we only mention \cite{Bainov_Mishev1991} and the references cited therein, since this book investigates
oscillation effects for nonlinear partial differential equations with delay that we observe also for \eqref{E1.1}.
Our parabolic delay equation is nonlinear and contains a nonlocal Pyragas type feedback term defined by a measure. To our best knowledge, an associated result on existence and uniqueness of a solution is not yet known.
\begin{theorem}
For every $u \in \MT$, problem \eqref{E1.1} has a unique solution $y_u \in Y$. Moreover, the estimates
\begin{align}
& \|y_u\|_{L^2(0,T;H^1(\Omega))} \le C_{1,0}\big(\|y_0\|_{L^2(Q_-)}\|u\|_\MT + \|y_0(\cdot,0)\|_{L^2(\Omega)} + |R(0)|\big),\label{E2.2}\\
&\|y_u\|_{C(\bar Q)} \le C_\infty\big(\|y_0\|_{C(\bar Q_-)}\|u\|_\MT +  \|y_0(\cdot,0)\|_{C(\bar\Omega)} + |R(0)|\big),\label{E2.3}
\end{align}
are satisfied, where the constants $C_{1,0}$ and $C_\infty$ depend on $\|u\|_\MT$, but they can be taken fixed on bounded subsets of $\MT$.
\label{T2.1}
\end{theorem}

In order to prove this theorem, we perform the classical substitution $y^\lambda(x,t) = \text{e}^{-\lambda t}y(x,t)$ with   arbitrary $\lambda > 0$. Hence equation \eqref{E1.1} is transformed to
\[
\left\{\begin{array}{rcll}\displaystyle\frac{\partial y^\lambda}{\partial t} - \Delta y^\lambda + \text{e}^{-\lambda t}R(\text{e}^{\lambda t}y^\lambda)  + \lambda y^\lambda&=&\displaystyle \int_0^T \text{e}^{-\lambda s}y^\lambda(x,t-s)\dus \ &\text{ in } Q,\vspace{2mm}\\ \displaystyle \partial_ny^\lambda &=& 0 \ &\text{ on } \Sigma ,\\[1ex] y^\lambda(x,t) &=&\text{e}^{-\lambda t}y_0(x,t) \ &\text{ in }Q_-. \end{array}\right.
\]
To simplify the notation we introduce, for every $\lambda \ge 0$, the following family of operators $K_\lambda[u], K_\lambda^+[u]:C(\bar Q) \longrightarrow \rli$, defined for 
{$(x,t) \in \bar Q$} by
\begin{align*}
&(K_\lambda[u]z)(x,t) = \int_{[0,t)}\text{e}^{-\lambda s}z(x,t-s)\dus,\\
&(K_\lambda^+[u]z)(x,t) = \int_{(0,t)}\text{e}^{-\lambda s}z(x,t-s)\dus.
\end{align*}
Hence, we have $(K_\lambda[u]z)(x,t) = u(\{0\})z(x,t) + (K_\lambda^+[u]z)(x,t)$.

Moreover, we define the family of functions
\[
g_{\lambda,u}(x,t) = \int_t^T\text{e}^{-\lambda s}y_0(x,t-s)\dus\ \ \text{ for } (x,t) \in \bar Q
\]
that covers the initial data $y_0$.

For $\lambda = 0$ we simply write $K[u]$ and $g_u$ instead of $K_0[u]$ and $g_{0,u}$, respectively. With this notation, the above equation and \eqref{E1.1} (obtained for $\lambda = 0$) can be formulated as follows
\begin{equation}
\left\{\begin{array}{rcll}\displaystyle\frac{\partial y^\lambda}{\partial t} - \Delta y^\lambda + \text{e}^{-\lambda t}R(\text{e}^{\lambda t}y^\lambda)  + \lambda y^\lambda&=&\displaystyle K_\lambda[u]y^\lambda + g_{\lambda,u} \ &\text{ in } Q,\vspace{2mm}\\ \displaystyle \partial_ny^\lambda &=& 0 \ &\text{ on } \Sigma ,\\[1ex] y^\lambda(x,0) &=&y_0(x,0) \ &\text{ in }\Omega, \end{array}\right.
\label{E2.4}
\end{equation}
with the additional extension $y^\lambda(x,t) = y_0(x,t)$ for $t \in [-T,0]$.

{
\begin{remark}
Notice that $K_{\lambda}[u] y$ and $g_{\lambda,u}$ can be discontinuous at those points $t$ such that $u(\{t\})\neq 0$, but the following identity holds
\[
\begin{aligned}
\int_0^T e^{-\lambda s}y^\lambda(x,t-s)du(s) &= \int_{[0,t)} e^{-\lambda s}y^\lambda(x,t-s)du(s) + \int_t^T e^{-\lambda s}y_0(x,t-s)du(s) \\
&= (K_\lambda[u]y^\lambda + g_{\lambda,u})(x,t),
\end{aligned}
\]
and, therefore, $K_{\lambda}[u] y + g_{\lambda,u}\in C(\bar Q)$.
\label{R2.2}
\end{remark}}

\begin{lemma}
For every $u \in \MT$ and $\lambda > 0$ we have
\begin{align}
&\left\{\begin{array}{lll}\displaystyle \|K_\lambda[u]z\|_{L^2(Q)} &\le \|z\|_{L^2(Q)}\|u\|_\MT& \forall z \in C(\bar Q),\\
\displaystyle \|K_\lambda[u]z\|_{
{L^\infty((0,T)},L^2(\Omega))} &\le  \|z\|_{C([0,T],L^2(\Omega))}\|u\|_\MT& \forall z \in C(\bar Q),\\
\displaystyle|K_\lambda[u]z\|_{\rli} &\le  \|z\|_{C(\bar Q)}\|u\|_\MT& \forall z \in C(\bar Q),\end{array}\right.\label{E2.5}\\
&\left\{\begin{array}{ll}\|g_{\lambda,u}\|_{L^2(Q)} &\le \|y_0\|_{L^2(Q_-)}\|u\|_\MT,\\
\|g_{\lambda,u}\|_{\rli} &\le \|y_0\|_{C(\bar Q_-)}\|u\|_\MT.\end{array}\right.
\label{E2.6}
\end{align}
Moreover, for every $\varepsilon > 0$ and $u \in \MT$ there exists $\lambda_{\varepsilon,u} > 0$ such that $\forall \lambda \ge \lambda_{\varepsilon,u}$ the following inequalities hold
\begin{equation}
\left\{\begin{array}{ll}\displaystyle \|K_\lambda^+[u]z\|_{L^2(Q)} &\le  \varepsilon\big(1 + \|u\|_\MT\big)\|z\|_{L^2(Q)}\quad \forall z \in C(\bar Q),\\ \displaystyle \|K_\lambda^+[u]z\|_{\rli} &\le  \varepsilon\big(1 + \|u\|_\MT\big)\|z\|_{C(\bar Q)}\quad \ \forall z \in C(\bar Q),\\
\displaystyle \|g_{\lambda,u}\|_{\rli} &\le \varepsilon\big(1 + \|u\|_\MT\big)\|y_0\|_{C(\bar Q_-)}.\end{array}\right.
\label{E2.7}
\end{equation}
\label{L2.1}
\end{lemma}

\begin{proof}
By using the Schwarz inequality and the Fubini theorem, we get
\begin{align*}
&\|K_\lambda[u]z\|_{L^2(Q)}^2 = \int_Q\left(\int_{[0,t)}\text{e}^{-\lambda s}z(x,t-s)\dus\right)^2\dx\dt\\
& \le \int_Q\left(\int_{[0,t)}z^2(x,t-s)\duas\right)\left(\int_{[0,t)}\text{e}^{-2\lambda s}\duas\right)\dx\dt\\
&\le \left(\int_\Omega \int_0^T\int_s^Tz^2(x,t-s)\dt\duas\dx\right)\left(\int_0^T\text{e}^{-2\lambda s}\duas\right) = I_1 I_2.
\end{align*}
Substituting $\sigma = t - s$ we get for $I_1$
\begin{align*}
&I_1 = \int_\Omega \int_0^T\int_0^{T-s}z^2(x,\sigma)\, d\sigma\duas\dx \le \int_\Omega \int_0^T\int_0^{T}z^2(x,\sigma)\, d\sigma\duas\dx\\
& = \|z\|^2_{L^2(Q)}\|u\|_\MT.
\end{align*}

To estimate $I_2$ we proceed as follows
\begin{align*}
&I_2 = \int_0^T\text{e}^{-2\lambda s}\duas \le \int_0^T\duas =\|u\|_\MT.
\end{align*}
Multiplying the estimates for $I_1$ and $I_2$, we get the first inequality of \eqref{E2.5}. To prove the second estimate we proceed as follows: $\forall t \in [0,T]$
\begin{align*}
&\|(K_\lambda[u]z)(\cdot,t)\|_{L^2(\Omega)}^2 = \int_\Omega\left(\int_{[0,t)}\text{e}^{-\lambda s}z(x,t-s)\dus\right)^2\dx\dt\\
&\le \left(\int_\Omega\int_{[0,t)}z^2(x,t-s)\duas\dx\right)\left(\int_0^T\text{e}^{-2\lambda s}\duas\right)\\
& = \left(\int_{[0,t)}\|z(\cdot,t-s)\|_{L^2(\Omega)}^2\duas\right)I_2 \le \|z\|^2_{C([0,T],L^2(\Omega))}\|u\|_\MT^2.
\end{align*}

Finally, to prove the third inequality of \eqref{E2.5} we only need the following estimation
\[
\|K_\lambda[u]z\|_{\rli} \le \|z\|_{C(\bar Q)}I_2 \le \|z\|_{C(\bar Q)}\|u\|_\MT.
\]
In a completely analogous way we prove \eqref{E2.6}. To establish the first two inequalities of \eqref{E2.7} we proceed exactly as above replacing $I_2$ by $I_2^+$
\begin{align}
& I_2^+ = \int_{(0,T]}\text{e}^{-2\lambda s}\duas \le \int_{(0,\frac{1}{\sqrt{\lambda}})}\text{e}^{-2\lambda s}\duas + \int_{\frac{1}{\sqrt{\lambda}}}^T\text{e}^{-2\lambda s}\duas\notag\\
& \le |u|\big(0,\frac{1}{\sqrt{\lambda}}\big) + \text{e}^{-2\sqrt{\lambda}}\|u\|_\MT.\label{E2.8}
\end{align}
It is enough to choose $\lambda_{\varepsilon,u} > 0$ sufficiently large such that
\[
|u|\big(0,\frac{1}{\sqrt{\lambda_{\varepsilon,u}}}\big) + \text{e}^{-\sqrt{\lambda_{\varepsilon,u}}} < \varepsilon
\]
holds to conclude the desired inequalities. Finally, we prove the last inequality of \eqref{E2.7}. To this end, we observe that
\[
|g_{\lambda,u}(x,t)| \le \int_{(0,T]}\text{e}^{-\lambda s}|y_0(x,t-s)|\, \duas \le I_1^+\|y_0\|_{C(\bar Q_-)}\ \ \forall (x,t) \in Q,
\]
where
\[
I_1^+ = \int_{(0,T]}\text{e}^{-\lambda s}\, \duas \le |u|\big(0,\frac{1}{\sqrt{\lambda}}\big) + \text{e}^{-\sqrt{\lambda}}\|u\|_\MT \le \varepsilon\big(1 + \|u\|_\MT\big).
\]
Combining this estimate with \eqref{E2.8} we deduce the desired inequality.
\end{proof}

{\em Proof of Theorem \ref{T2.1}.}  We split the proof into three steps.

{\em I - Existence of a solution.} For every function $z \in C(\bar Q)$, we define the problem
\begin{equation}
\left\{\begin{array}{rcll}\displaystyle\frac{\partial y}{\partial t} - \Delta y + R^+_\lambda(t,y) &=&K_\lambda^+[u]z + g_{\lambda,u} \ &\text{ in } Q,\vspace{2mm}\\ \displaystyle \partial_ny &=& 0 \ &\text{ on } \Sigma ,\\[1ex] y(x,0) &=&y_0(x,0) \ &\text{ in }\Omega, \end{array}\right.
\label{E2.9}
\end{equation}
where $R^+_\lambda(t,y) = \text{e}^{-\lambda t}R(\text{e}^{\lambda t}y) + (\lambda - u(\{0\}))y$. This is a standard semilinear parabolic equation with given right hand side. We have that $\frac{\partial R^+_\lambda}{\partial y}(t,y) \ge C_R + \lambda - u(\{0\})$ and we set $\lambda_R^+ = 1 - \min\{0,C_R - u(\{0\})\}$,  hence it follows that
\begin{equation} \label{{E2.10}}
\frac{\partial R_\lambda^+}{\partial y}(t,y) \ge 1 \quad \forall \lambda \ge \lambda_R^+.
\end{equation}

Therefore, $R_\lambda^+$ is a continuous monotone increasing function with respect to $y$, and it is well known that the semilinear equation \eqref{E2.9} has a unique solution $y \in Y$; see \cite{Casas97} or \cite{Troltzsch2010}, for instance. The continuity is due to the continuity of $y_0(\cdot,0)$ and the fact that the right hand side of the partial differential equation in \eqref{E2.9} belongs to $\rli$. Moreover from the above references and the equality $R_\lambda^+(t,0) = R(0)$ we know the estimate
\begin{equation}
\|y\|_{C(\bar Q)} \le C_0\Big(\|K_\lambda^+[u]z\|_{\rli} + \|g_{\lambda,u}\|_{\rli} +  \|y_0(\cdot,0)\|_{C(\bar\Omega)} +  |R(0)|\Big).
\label{E2.11}
\end{equation}

Now, we define $M = C_0(1 +  \|y_0(\cdot,0)\|_{C(\bar\Omega)} + |R(0)|)$. According to \eqref{E2.7}, we can select $\lambda \ge \lambda_R^+$ such that
\begin{equation}
\|K_\lambda^+[u]z\|_{\rli} + \|g_{\lambda,u}\|_{\rli} \le 1\ \ \forall z \in C(\bar Q)\ \text{ with } \|z\|_{C(\bar Q)} \le M.
\label{E2.12}
\end{equation}
Let $B_M$ be the closed ball of $C(\bar Q)$ with center at $0$ and radius $M$. We define the continuous mapping $F:B_M \longrightarrow B_M$ that associates to every $z \in B_M$ the solution $y = F(z)$ of \eqref{E2.9}. The embedding $F(B_M) \subset B_M$ is an immediate consequence of \eqref{E2.11}, \eqref{E2.12} and the definition of $M$. In order to apply Schauder's fixed point  theorem, we have to prove that $F(B_M)$ is relatively compact in $C(\bar Q)$. To this end,  we assume first that $y_0(\cdot,0)$ is a H\"older function in $\bar \Omega$: $y_0(\cdot, 0) \in C^{0,\mu}(\bar\Omega)$ with $\mu \in (0,1)$. Then, there exists $\mu_0 \in (0,\mu]$ and $C_\mu$ such that $y \in C^{0,\mu_0}(\bar Q)$ and
\[
\|y\|_{C^{0,\mu_0}(\bar Q)} \le C_{\mu}\Big(\|K_\lambda^+[u]z + g_{\lambda,u}\|_{\rli} +  \|y_0(\cdot,0)\|_{C^{0,\mu}(\bar\Omega)}+ \max_{|\rho|\le M} |R(\rho)| + 1\Big);
\]
see \cite[\S III-10]{Lad-Sol-Ura68}. From the compactness of the embedding $C^{0,\mu_0}(\bar Q) \subset C(\bar Q)$ and the above estimate we conclude that $F(B_M)$ is relatively compact in $C(\bar Q)$ and $F$ has at least one fixed point $y^\lambda$. Then it is obvious that $y^\lambda$ is a solution of  \eqref{E2.4} and $y^\lambda \in Y$.

Next we skip the assumption $y_0(\cdot,0) \in C^{0,\mu}(\bar\Omega)$ for every sufficiently large $\lambda$. Since $y_0 \in C(\bar\Omega)$, then we can take a sequence $\{y_{0k}\}_{k = 1}^\infty \subset C^{0,\mu}(\bar\Omega)$ such that $y_{0k} \to y_0(\cdot,0)$ in $C(\bar\Omega)$. Hence, for every $k \ge 1$, \eqref{E2.4} has at least one solution $y^\lambda_k \in Y$. Let us prove that $\{y^\lambda_k\}_{k = 1}^\infty$ is a Cauchy sequence in $Y$. To this end we select two terms of the sequence $y^\lambda_k$ and $y^\lambda_m$ and subtract the equations satisfied by them. Then we get
\[
\left\{\begin{array}{rcll}\displaystyle\frac{\partial (y^\lambda_k - y^\lambda_m)}{\partial t} - \Delta (y^\lambda_k - y^\lambda_m) + \frac{\partial R_\lambda^+}{\partial y}(t,y^\lambda_{km})(y^\lambda_k - y^\lambda_m) &=&K_\lambda^+[u](y^\lambda_k - y^\lambda_m)&\text{ in } Q,\vspace{2mm}\\ \displaystyle \partial_n(y^\lambda_k - y^\lambda_m) &=& 0 \ &\text{ on } \Sigma ,\\[1ex] (y^\lambda_k - y^\lambda_m)(x,0) &=&y_{0k}(x) - y_{0m}(x) &\text{ in }\Omega, \end{array}\right.
\]
where $y^\lambda_{km} = y^\lambda_k + \theta_{k}(y^\lambda_k - y^\lambda_m)$ and $0 \le \theta_{km}(x,t) \le 1$ is a measurable function. Now, multiplying the above equation by $y^\lambda_k - y^\lambda_m$, using that $\frac{\partial R_\lambda}{\partial y}(t,y^\lambda_{km}) \ge 1$ due to our choice of $\lambda$ (cf. \eqref{{E2.10}}), and \eqref{E2.7}, we obtain
\begin{align*}
&\frac{1}{2}\|(y^\lambda_k - y^\lambda_m)(T)\|_{L^2(\Omega)} + \int_Q|\nabla(y^\lambda_k - y^\lambda_m)|^2\dx\dt + \int_Q|y^\lambda_k - y^\lambda_m|^2\dx\dt\\
& \le \varepsilon\big(1 + \|u\|_\MT\big)\|y^\lambda_k - y^\lambda_m\|^2_{L^2(Q)} + \frac{1}{2}\|y_{0k} - y_{0m}\|_{L^2(\Omega)}^2\\
& \le C\varepsilon\big(1 + \|u\|_\MT\big)\|y^\lambda_k - y^\lambda_m\|^2_{L^2(0,T;H^1(\Omega))} + \frac{1}{2}\|y_{0k} - y_{0m}\|_{L^2(\Omega)}^2.
\end{align*}
For all $\varepsilon$ sufficiently small and $\lambda \ge \lambda_{\varepsilon,u}$, this leads to
\[
\|y^\lambda_k - y^\lambda_m\|_{L^2(0,T;H^1(\Omega))} \le C_1\|y_{0k} - y_{0m}\|_{L^2(\Omega)}.
\]
Notice that the left hand side of the above chain of inequalities absorbs the term appearing with the factor $C\varepsilon\big(1 + \|u\|_\MT\big)$ in the right hand side, if $\varepsilon$ is small enough.

Hence $\{y^\lambda_k\}_{k = 1}^\infty$ is a Cauchy sequence in $L^2(0,T;H^1(\Omega))$. To prove that it is a Cauchy sequence in $C(\bar Q)$ as well, we handle the equation satisfied by $y^\lambda_k - y^\lambda_m$ in the same way as we discussed \eqref{E2.9}. We use \eqref{E2.7} to get
\[
\|y^\lambda_k - y^\lambda_m\|_{C(\bar Q)} \le C_0\big(\varepsilon\big(1 + \|u\|_\MT\big)\|y^\lambda_k - y^\lambda_m\|_{C(\bar Q)} + \|y_{0k} - y_{0m}\|_{C(\bar\Omega)}\big).
\]
Taking again $\varepsilon$ sufficiently small and $\lambda \ge \lambda_{\varepsilon,u}$, we deduce
\[
\|y^\lambda_k - y^\lambda_m\|_{C(\bar Q)} \le C_2\|y_{0k} - y_{0m}\|_{C(\bar\Omega)}.
\]
Therefore, $\{y^\lambda_k - y^\lambda_m\}_{k = 1}^\infty$ is a Cauchy sequence in $C(\bar Q)$. Consequently, there exists $y^\lambda \in Y$ such that $y^\lambda_k \to y^\lambda$ in $Y$. It is easy to see that $y^\lambda$ is a solution of \eqref{E2.4}. Now, we  re-substitute $y_u(x,t) = \text{e}^{\lambda t}y^\lambda(x,t)$ and extend $y_u$ to $\bar Q_-$ by $y_0$ to get that $y_u$ is a solution of \eqref{E1.1}.

{\em II - Uniqueness of the solution.} Let $y^\lambda_1, y^\lambda_2 \in Y$ be two solutions of \eqref{E2.4}, and set $y^\lambda = y^\lambda_2 - y^\lambda_1$. Subtracting the equations satisfied by $y^\lambda_2$ and $y^\lambda_1$ we obtain
\[
\left\{\begin{array}{rcll}\displaystyle\frac{\partial y^\lambda}{\partial t} - \Delta y^\lambda + \frac{\partial R_\lambda^+}{\partial y}(t,\hat y^\lambda)y^\lambda&=&\displaystyle K_\lambda^+[u]y^\lambda \ &\text{ in } Q,\vspace{2mm}\\ \displaystyle \partial_ny^\lambda &=& 0 \ &\text{ on } \Sigma ,\\[1ex] y^\lambda(x,0) &=&0\ &\text{ in }\Omega, \end{array}\right.
\]
where $\hat y^\lambda = y^\lambda_1 + \theta(y^\lambda_2 -  y^\lambda_1)$ is some intermediate state with $0 \le \theta(x,t) \le 1$. Multiplying this equation by $y^\lambda$ and invoking again \eqref{{E2.10}}
along with \eqref{E2.5}, we obtain
\begin{align*}
&\frac{1}{2}\|y^\lambda(T)\|^2_{L^2(\Omega)} + \int_Q|\nabla y^\lambda|^2\dx\dt + \int_Q|y^\lambda|^2\dx\dt\\
& \le \varepsilon\big(1 + \|u\|_\MT\big)\|y^\lambda\|^2_{L^2(Q)}.
\end{align*}
Taking $\varepsilon < \big(1 + \|u\|_\MT\big)^{-1}$, we conclude for $\lambda \ge \lambda_{\varepsilon,u}$ that $y^\lambda = 0$, since the last term in the left-hand side absorbs the right-hand side. Obviously the uniqueness of solution of \eqref{E2.4} is equivalent to the uniqueness of solution of \eqref{E1.1}.

{\em III - Estimates.} First we recall that $y_u(x,t) = e^{\lambda t} y^\lambda(x,t)$ is the solution of \eqref{E1.1}, once it has been extended to $\bar Q_-$ by $y_0$. Moreover, the following inequalities hold
\[
\|y_u\|_{L^2(0,T;L^2(\Omega))} \le \text{e}^{\lambda T}\|y^\lambda\|_{L^2(0,T;L^2(\Omega))} \ \text{ and }\ \|y_u\|_{C(\bar Q)} \le \text{e}^{\lambda T}\|y^\lambda\|_{C(\bar Q)}.
\]
Therefore it is enough to establish the estimates for $y^\lambda$. To this end, we define this time $R_\lambda(t,y) = \text{e}^{-\lambda t}R(\text{e}^{\lambda t}y) + \lambda y$ with
\[
\lambda \ge \lambda_R = 2(1 + \|u\|_\MT) - \min\{0,C_R\}.
\]
Now, we multiply equation \eqref{E2.4} by $y^\lambda$ and deal with the reaction term as follows:
\begin{align*}
&R_\lambda(t,y^\lambda)y^\lambda = [R_\lambda(t,y^\lambda) - R_\lambda(t,0)]y^\lambda + R_\lambda(t,0)y^\lambda\\
&= \frac{\partial R_\lambda}{\partial y}(t,\theta y^\lambda)|y^\lambda|^2 + R_\lambda(t,0)y^\lambda\ge 2(1 + \|u\|_\MT)|y^\lambda|^2 - |R(0)||y^\lambda|.
\end{align*}
Then, multiplying equation \eqref{E2.4} by $y^\lambda$ and using this inequality along with \eqref{E2.5} and \eqref{E2.6}, we obtain for every $0< T' < T$ and $Q_{T'} = \Omega \times (0,T')$
\begin{align*}
&\frac{1}{2}\|y^\lambda(T')\|^2_{L^2(\Omega)} + \int_{Q_{T'}}|\nabla y^\lambda|^2\dx\dt + 2(1 + \|u\|_\MT)\int_{Q_{T'}}|y^\lambda|^2\dx\dt\\
&\le \int_{Q_{T'}}(K_\lambda[u]y^\lambda + g_{\lambda,u})y^\lambda\dx\dt + |R(0)|\int_{Q_{T'}}|y^\lambda|\dx\dt + \frac{1}{2}\|y_0(\cdot,0)\|_{L^2(\Omega)}^2\\
&\le \|u\|_\MT\big(\|y^\lambda\|_{L^2(Q_{T'})} + \|y_0\|_{L^2(Q_-)})\|y^\lambda\|_{L^2(Q_{T'})}\\
& + |R(0)||Q|^{\frac{1}{2}}\|y^\lambda\|_{L^2(Q_{T'})} + \frac{1}{2}\|y_0(\cdot,0)\|_{L^2(\Omega)}^2\\
&\le    \big(\frac{1}{2} + \|u\|_\MT\big)\big\|y^\lambda\|^2_{L^2(Q_{T'})} +  \frac{1}{2}\|u\|_\MT^2\big\|y_0\|^2_{L^2(Q_-)}\\
& +  \frac{|Q|}{2}|R(0)|^2 + \frac{1}{2}\|y^\lambda\|^2_{L^2(Q_{T'})} + \frac{1}{2}\|y_0(\cdot,0)\|_{L^2(\Omega)}^2\\
& = \big(1 + \|u\|_\MT\big)\big\|y^\lambda\|^2_{L^2(Q_{T'})} +  \frac{1}{2}\|u\|_\MT^2\big\|y_0\|^2_{L^2(Q_-)}\\
& +  \frac{|Q|}{2}|R(0)|^2 + \frac{1}{2}\|y_0(\cdot,0)\|_{L^2(\Omega)}^2.
\end{align*}
The first term of the right hand side can be absorbed by the left hand side. In this way, we get
\begin{align}
&\|y^\lambda\|_{L^2(0,T;H^1(\Omega))} + \|y^\lambda\|_{C([0,T];L^2(\Omega))}\notag\\
& \le C\big(\|u\|_\MT^2\big\|y_0\|_{L^2(Q_-)} +  |R(0)| +\|y_0(\cdot,0)\|_{L^2(\Omega)}\big). \label{E2.13}
\end{align}

To prove \eqref{E2.3} we use the second inequality of \eqref{E2.5}, \eqref{E2.6} and the results of \cite[\S III-7]{Lad-Sol-Ura68} applied to \eqref{E2.4} to obtain
\begin{align*}
&\|y^\lambda\|_{C(\bar Q)} \le C\Big(\|K_\lambda^+[u]y^\lambda\|_{
{L^\infty((0,T)},L^2(\Omega))} + \|g_{\lambda,u}\|_{\rli} +  \|y_0(\cdot,0)\|_{C(\bar\Omega)} +  |R(0)|\Big)\\
& \le C\Big(\|u\|_\MT\|y^\lambda\|_{C([0,T],L^2(\Omega))} + \|u\|_\MT\|y_0\|_{C(Q_-)} +  \|y_0(\cdot,0)\|_{C(\bar\Omega)} +  |R(0)|\Big).
\end{align*}
Finally, from the equation satisfied by $y^\lambda$, the above estimates and the identity $y(x,t) = \text{e}^{\lambda t}y^\lambda(x,t)$ we conclude \eqref{E2.2} and \eqref{E2.3}.
\endproof

Let us prove some extra regularity of the solution of \eqref{E1.1}.

\begin{theorem}
Under the assumptions of Theorem \ref{T2.1}, if $y_0(\cdot,0) \in H^1(\Omega)$, then $y_u \in H^1(Q)$ and
\begin{equation}
\|y_u\|_{H^1(Q)} \le C_{1,1}\big(\|y_0\|_{L^2(Q_-)}\|u\|_\MT + \|y_0(\cdot,0)\|_{H^1(\Omega)} + |R(0)|\big).
\label{E2.14}
\end{equation}
In addition, if either $\Omega$ is convex or $\Gamma$ is of class $C^{1,1}$, then $y_u \in H^{2,1}(Q)$ and
\begin{equation}
\|y_u\|_{H^{2,1}(Q)} \le C_{2,1}\big(\|y_0\|_{L^2(Q_-)}\|u\|_\MT + \|y_0(\cdot,0)\|_{H^1(\Omega)} + |R(0)|\big).
\label{E2.15}
\end{equation}
The constants $C_{1,1}$ and $C_{2,1}$ depend on $\|u\|_\MT$, but they can be kept fixed on bounded subsets of $\MT$.
\label{T2.2}
\end{theorem}

\begin{proof}
For the first part of the theorem we only have to prove that $\frac{\partial y}{\partial t}$ belongs to $L^2(Q)$ and to confirm
the associated estimate. This is a simple consequence of a result that is known for linear parabolic equations; see, for instance, \cite[\S III.2]{Showalter1997}. Indeed, it is enough to write the equation in the form
\[
\frac{\partial y}{\partial t} - \Delta y = K[u]y + g_u - R(y).
\]
Thanks to \cite[\S III.2]{Showalter1997}, the $H^1(Q)$-norm
of $y_u$ can be estimated against the $L^2(Q)$-norm of the right hand side, and additionally $y_u \in L^2(0,T;D(\Delta))$ holds. Therefore, if $\Omega$ is convex or $\Gamma$ is of class $C^{1,1}$, then $D(\Delta) = H^2(\Omega)$ and the estimate \eqref{E2.15} follows; see \cite[Chapters 2 and 3]{Grisvard85}.
\end{proof}

The next step of our analysis is the investigation of the differentiability properties of the control-to-state
mapping $G:\MT \longrightarrow Y$ that associates to $u \in \MT$ the solution $y_u$ of  \eqref{E1.1}, $G(u) = y_u$.

\begin{theorem}
The mapping $G$ is of class $C^1$. For every $u,v \in \MT$, we have that $z_v = G'(u)v$ is the solution of the  problem
\begin{equation}
\left\{\begin{array}{rcll}\displaystyle\frac{\partial z}{\partial t} - \Delta z + R'(y_u)z  &=&\displaystyle K[u]z  +  K[v]y_u + g_v \ &\text{ in } Q,\vspace{2mm}\\ \displaystyle \partial_nz &=& 0 \ &\text{ on } \Sigma ,\\[1ex] z(x,0) &=&0 \ &\text{ in }\Omega. \end{array}\right.
\label{E2.16}
\end{equation}
\label{T2.3}
\end{theorem}

\begin{proof}
We define the space
\[
\mathcal{Y} = \Big\{y \in Y : \frac{\partial y}{\partial t} - \Delta y \in \rli \Big\},
\]
endowed with the norm
\[
\|y\|_{\mathcal{Y}} = \|y\|_Y + \|\frac{\partial y}{\partial t} - \Delta y\|_{\rli},
\]
$\mathcal{Y}$ is a Banach space. Now we consider the mapping
\begin{align*}
&\mathcal{F} : \mathcal{Y} \times \MT \longrightarrow \rli \times C(\bar\Omega),\\
&\mathcal{F}(y,u) = \big(\frac{\partial y}{\partial t} - \Delta y + R(y) - K[u]y - g_u,y(\cdot,0) - y_0(\cdot,0)\big).
\end{align*}
It is obvious that $\mathcal{F}$ is well defined and
is of class $C^1$. Moreover, we have that
\[
\frac{\partial\mathcal{F}}{\partial y}(y,u)z = \Big(\frac{\partial z}{\partial t} - \Delta z + R'(y)z - K[u]z,z(\cdot,0)\Big).
\]
Let us confirm that $\frac{\partial\mathcal{F}}{\partial y}(y,u):\mathcal{Y} \longrightarrow \rli \times C(\bar\Omega)$ is an isomorphism. Indeed, since obviously $\frac{\partial\mathcal{F}}{\partial y}(y,u)$ is a linear and continuous mapping, we only need to prove that, for every pair $(f,z_0) \in \rli  \times C(\bar\Omega)$, there exists a unique solution $z \in \mathcal{Y}$ of the problem
\[
\left\{\begin{array}{rcll}\displaystyle\frac{\partial z}{\partial t} - \Delta z + R'(y)z  &=&\displaystyle K[u]z + f \ &\text{ in } Q,\vspace{2mm}\\ \displaystyle \partial_nz &=& 0 \ &\text{ on } \Sigma ,\\[1ex] z(x,0) &=&z_0 \ &\text{ in }\Omega. \end{array}\right.
\]
The existence and uniqueness of such a solution is proved in the same way as for the problem \eqref{E1.1}. Hence, an application of the implicit function theorem implies that $G$ is of class $C^1$. The equation \eqref{E2.16} follows easily by differentiating the identity $\mathcal{F}(G(u),u) = 0$ with respect to $u$.
\end{proof}

\begin{remark}
Let us mention  that $z_v = G'(u)v \in H^1(Q)$ holds for every $v \in \MT$. This follows from equation \eqref{E2.16} arguing similarly as in the proof of Theorem \ref{T2.2} and taking into account that $z_v(x,0) = 0$.
\label{R3.1}
\end{remark}

\section{The Control Problem}
\label{S4}
\setcounter{equation}{0}

Now we have all prerequisites to study our optimal control problem, namely

\[
\Pb \min_{u \in \MT}  J(u) = \frac{1}{2}\int_Q|y_u - y_d|^2\dx\dt + \nu\|u\|_\MT,
\]
where $y_d \in L^{\bar p}(Q)$ for some $\bar p > 1 + \frac{n}{2}$ and $\nu > 0$ are given.

\begin{theorem}
Problem \Pb has at least one solution $\bar u$.
\label{T3.1}
\end{theorem}

Before proving this theorem we state the following lemma.

\begin{lemma}
Assume that $u_k \stackrel{*}{\rightharpoonup} \bar u$ in $\MT$ for $k \to \infty$ and let $y_k$ and $\bar y$ be the states associated with $u_k$ and $\bar u$, respectively; then $y_k \to \bar y$ in $Y$.
\label{L3.1}
\end{lemma}
\begin{proof}
Since $u_k \stackrel{*}{\rightharpoonup} \bar u$ in $\MT$  as $k \to \infty$, we know that there exists a constant $M > 0$ such that $\|u_k\|_\MT \le M$ $\forall k \ge 1$. Hence, $\|\bar u\|_\MT \le M$ holds as well. Set $y^\lambda_k(x,t) = \text{e}^{-\lambda t}y_k(x,t)$ and $\bar y^\lambda(x,t) = \text{e}^{-\lambda t}\bar y(x,t)$. Then $y^\lambda_k$ and $\bar y^\lambda$ satisfy \eqref{E2.4} for controls $u:=u_k$ and $u:=\bar u$, respectively. Let us define $w_k^\lambda = y^\lambda_k - \bar y^\lambda$. Then, subtracting these two equations and taking again $R_\lambda(t,y) = \text{e}^{-\lambda t}R(\text{e}^{\lambda t}y) + \lambda y$ with $\lambda \ge 2(M+1) - \min\{0,C_R\}$, we get
\begin{equation}
\left\{\begin{array}{rcll}\displaystyle\frac{\partial w^\lambda_k}{\partial t} - \Delta w^\lambda_k + \frac{\partial R_\lambda}{\partial y}(t,\hat y^\lambda_k)w^\lambda_k&=&\displaystyle K_\lambda[u_k]y_k^\lambda  - K_\lambda[\bar u]\bar y^\lambda + g_{\lambda,u_k} - g_{\lambda,\bar u}\ &\text{ in } Q,\vspace{2mm}\\ \displaystyle \partial_nw_k^\lambda &=& 0 \ &\text{ on } \Sigma ,\\[1ex] w_k^\lambda(x,0) &=&0\ &\text{ in }\Omega. \end{array}\right.
\label{E3.1}
\end{equation}
with intermediate states $\hat y^\lambda_k$.

Testing this equation by $w^\lambda_k$ and invoking \eqref{E2.7}, we get for every $0 < T' < T$
\begin{align*}
&\frac{1}{2}\|w^\lambda_k(T')\|^2_{L^2(\Omega)} + \int_{Q_{T'}} \left[|\nabla w^\lambda_k|^2 + 2(M+1)|w^\lambda_k|^2\right]\dx\dt\\
&\le \int_{Q_{T'}}\left[K_\lambda[u_k]y_k^\lambda - K_\lambda[\bar u]\bar y + g_{\lambda,u_k} - g_{\lambda,\bar u}\right]w^\lambda_k\dx\dt\\
&= \int_{Q_{T'}}\left[K_\lambda[u_k]w^\lambda_k + K_\lambda[u_k - \bar u]\bar y + g_{\lambda,u_k - \bar u}\right]w^\lambda_k\dx\dt\\
&\le \|u_k\|_\MT\|w^\lambda_k\|^2_{L^2(Q_{T'})} + \|K_\lambda[u_k - \bar u]\bar y + g_{\lambda,u_k - \bar u}\|_{L^2(Q)}\|w^\lambda_k\|_{L^2(Q_{T'})}\\
&\le (M+1)\|w^\lambda_k\|_{L^2(Q_{T'})} + \frac{1}{2}\|K_\lambda[u_k - \bar u]\bar y + g_{\lambda,u_k - \bar u}\|_{L^2(Q)}^2.
\end{align*}

The first term of the right hand side can be absorbed by the left hand side and we infer
\begin{equation}
\|w^\lambda_k\|_{L^2(0,T;H^1(\Omega))} + \|w^\lambda_k\|_{C([0,T],L^2(\Omega))} \le C\big(\|K_\lambda[u_k - \bar u]\bar y + g_{\lambda,u_k - \bar u}\|_{L^2(Q)}\big).
\label{E3.2}
\end{equation}
Let us prove that the right hand side of the inequality converges to zero. From the convergence $u_k \stackrel{*}{\rightharpoonup} \bar u$ in $\MT$ and by the continuity of $\bar y$ we get for $k \to \infty$
\[
(K_\lambda[u_k - \bar u]\bar y + g_{\lambda,u_k - \bar u}) (x,t) = \int_0^T\text{e}^{-\lambda s}\bar y(x,t-s)\, d(u_k - \bar u)(s) \to 0 \ \ \forall (x,t) \in Q,
\]
i.e. pointwise convergence. Moreover, from \eqref{E2.5} and \eqref{E2.6} we have
\begin{align*}
&\|K_\lambda[u_k - \bar u]\bar y + g_{\lambda,u_k - \bar u}\|_{\rli}\le \big(\|\bar y\|_{C(\bar Q)} + \|y_0\|_{C(\bar Q_-)}\big)\|u_k - \bar u\|_\MT\\
& \le 2M\big(\|\bar y\|_{C(\bar Q)} + \|y_0\|_{C(\bar Q_-)}\big)\quad \forall k.
\end{align*}
From the Lebesgue dominated convergence theorem we conclude that $K_\lambda[u_k - \bar u]\bar y + g_{\lambda,u_k - \bar u} \to 0$ in $L^p(Q)$ for every $p < \infty$. Therefore, we infer from \eqref{E3.2} the convergence $w_k^\lambda \to 0$ in $L^2(0,T;H^1(\Omega)) \cap C([0,T],L^2(\Omega))$, and hence $y_k \to \bar y$ in $L^2(0,T;H^1(\Omega)) \cap C([0,T],L^2(\Omega))$.

Let us show show the uniform convergence.  From equation \eqref{E3.1}, using the estimates of  \cite[\S III-8]{Lad-Sol-Ura68} and \eqref{E2.5} we infer for $p > 1 + \frac{n}{2}$
\begin{align*}
&\|w^\lambda_k\|_{C(\bar Q)} \le C_1\|K_\lambda[u_k]w^\lambda_k\|_{
{L^\infty((0,T)},L^2(\Omega))} + C_2\|K_\lambda[u_k - \bar u]\bar y + g_{\lambda,u_k - \bar u}\|_{L^p(\Omega))}\\
&\le C_1M\|w^\lambda_k\|_{C([0,T],L^2(\Omega))} + C_2\|K_\lambda[u_k - \bar u]\bar y + g_{\lambda,u_k - \bar u}\|_{L^p(\Omega))} \to 0 \text{ as } k \to \infty.
\end{align*}
We have proved that $w_k^\lambda \to 0$ in $Y$. Transforming $y_k^\lambda$ and $\bar y^\lambda$ back to $y_k$ and $\bar y$, this leads to $\|y_k - \bar y\|_Y \to 0$.
\end{proof}

{\em Proof of Theorem \ref{T3.1}.}  Let $\{u_k\}_{k = 1}^\infty \subset \MT$ be a minimizing sequence of \Pb. Since
\[
\nu\|u_k\|_\MT \le J(u_k) \le J(0) < + \infty,
\]
we deduce that $\{u_k\}_{k = 1}^\infty$ is bounded in $\MT$. Hence, we can extract a subsequence, denoted in the same way, such that $u_k \stackrel{*}{\rightharpoonup} \bar u$ in $\MT$. Denote by $y_k$ and $\bar y$ the states associated with $u_k$ and $\bar u$, respectively. From Lemma \ref{L3.1} we know that $y_k \to \bar y$ in $L^2(Q)$. This convergence, along with \eqref{E2.3}, implies that $J(\bar u) \le \liminf_{k \to \infty}J(u_k) = \inf\Pb$, and hence $\bar u$ is a solution of \Pb.
\endproof

Next we derive the first order optimality conditions that have to be satisfied by any local solution of the problem \Pb. We distinguish between two different types of local solutions. To this end, we recall that $\MT \subset H^1(0,T)^*$, the embedding being continuous and compact. Notice that $H^1(0,T)$ is compactly embedded in $C[0,T]$ and then by transposition we deduce the compactness of $\MT \subset H^1(0,T)^*$.

\begin{definition}
A control $\bar u$ is called a local solution or local minimum of  \Pb in the sense of $\MT$ (respectively $H^1(0,T)^*$) if there exists a ball $B_\varepsilon(\bar u)$ in the associated space such that $J(\bar u) \le J(u)$ $\forall u \in \MT \cap B_\varepsilon(\bar u)$.
We will say that $\bar u$ is a local solution if it is a local solution in some of the two notions defined above.
\label{D4.1}
\end{definition}

Due to the continuity of the above embeddings, it follows immediately that, if $\bar u$ is a local solution in the $H^1(0,T)^*$ sense, then it is also a local solution in the $\MT$  sense. The converse implication is not true, in general.

Let us define the two different functionals forming $J(u) = F(u) + \nu j(u)$ by
\[
F(u) = \frac{1}{2}\int_Q|y_u - y_d|^2\dx\dt \quad \text{ and } \quad
 j(u) = \|u\|_\MT.
\]

\begin{theorem}
The functional $F:\MT \longrightarrow \mathbb{R}$ is of class $C^1$. Its derivative is given by
\begin{equation}
F'(u)v = \int_Q\varphi_u(K[v]y_u + g_v)\, dx\dt\quad \forall u,v \in \MT,
\label{E3.3}
\end{equation}
where $\varphi_u \in H^1(Q) \cap C(\bar Q)$ is the solution of the adjoint state equation
\begin{equation}
\left\{\begin{array}{rcll}\displaystyle -\frac{\partial\varphi}{\partial t} - \Delta\varphi + R'(y_u)\varphi &=&\displaystyle K^*[u]\varphi + y_u - y_d \ &\text{ in } Q,\vspace{2mm}\\ \displaystyle \partial_n\varphi &=& 0 \ &\text{ on } \Sigma\\[1ex]\varphi(x,T) &=& 0 \ &\text{ in } \Omega, \end{array}\right.
\label{E3.4}
\end{equation}
and the operator $K^*$ is defined by
\begin{equation}
(K^*[u]w)(x,t) = \int_{[0,T - t)}w(x,t+s)\dus\quad \forall w \in C(\bar Q).
\label{E3.5}
\end{equation}
\label{T3.2}
\end{theorem}

Before proving this theorem we analyze the adjoint state equation \eqref{E3.4}.

\begin{proposition} For all $u \in \MT$,
there exists a unique solution $\varphi \in H^1(Q) \cap C(\bar Q)$ of \eqref{E3.4} and it holds
\begin{align}
&\|\varphi\|_{H^1(Q)} \le M_{1,1}\|y_u - y_d\|_{L^2(Q)},\label{E3.6}\\
&\|\varphi\|_{C(\bar Q)} \le M_\infty\|y_u - y_d\|_{L^{\bar p}(Q)}.\label{E3.7}
\end{align}
Moreover, if either $\Gamma$ is of class $C^{1,1}$ or $\Omega$ is convex, then $\varphi \in H^{2,1}(Q)$ and
\begin{equation}
\|\varphi\|_{H^{2,1}(Q)} \le M_{2,1}\|y_u - y_d\|_{L^2(Q)}.
\label{E3.8}
\end{equation}
The constants $M_{1,1}$, $M_\infty$ and $M_{2,1}$ depend on $u$, but they can be taken fixed on bounded subsets of $\MT$.
\label{P3.1}
\end{proposition}

\begin{proof}
Given $\lambda > 0$, we set $\psi^\lambda(x,t) = \text{e}^{-\lambda t} \varphi(x,T-t)$ in $Q$. Then we have $(K_\lambda[u]\psi)(x,t) = (\text{e}^{-\lambda t}K^*[u]\varphi)(x,T-t)$, and \eqref{E3.4} is transformed to the forward equation
\begin{equation}
\left\{\begin{array}{rcll}\displaystyle \frac{\partial\psi^\lambda}{\partial t} - \Delta\psi^\lambda + R'(\hat y_u)\psi^\lambda + \lambda\psi^\lambda &=&\displaystyle K_\lambda[u]\psi + f \ &\text{ in } Q,\vspace{2mm}\\ \displaystyle \partial_n\psi^\lambda &=& 0 \ &\text{ on } \Sigma\\[1ex]\psi^\lambda(x,0) &=& 0 \ &\text{ in } \Omega, \end{array}\right.
\label{E3.9}
\end{equation}
where $\hat y_u(x,t) = y_u(x,T-t)$ and $f(x,t) = (y_u - y_d)(x,T-t)$. Now, we can argue as in Theorems \ref{T2.1} and \ref{T2.2} to get the existence, uniqueness and regularity. The only difference is that $f \in L^{\bar p}(Q)$ with $\bar p > 1 + \frac{n}{2}$, which is enough to deduce the H\"older regularity of the solution of \eqref{E3.9}; see \cite[\S III-10]{Lad-Sol-Ura68}.
\end{proof}

{Let us observe that, in some sense, the operator $K^*[u]$ is the adjoint of $K[u]$ with respect to the $L^2(Q)$ scalar product.} Indeed, given $w, z \in C(\bar Q)$, applying Fubini's Theorem and making the change of variables $\tau = t + s$ we get
\begin{align}
&\int_Q (K^*[u]w)(x,t)z(x,t)\dx\dt = \int_Q\left(\int_{[0,T-t)}w(x,t+s)\dus\right) z(x,t)\dx\dt\notag\\
&=\int_\Omega\int_{[0,T)}\left(\int_0^{T-s}w(x,t+s)z(x,t)\dt\right)\dus\dx\notag\\
& = \int_\Omega\int_{[0,T)}\left(\int_s^Tw(x,\tau)z(x,\tau - s)\dtau\right)\dus\dx\notag\\
&= \int_\Omega\int_0^T\left(\int_{[0,\tau)} z(x,\tau-s)\dus\right)w(x,\tau)\dtau\dx\notag\\
& = \int_Q(K[u]z)(x,\tau)w(x,\tau)\dx\dtau.
\label{E3.10}
\end{align}

{\em Proof of Theorem \ref{T3.2}.} Let us set $z_v = G'(u)v$. Thanks to Remark \ref{R3.1} and Proposition \ref{P3.1}, we have that $z_v, \varphi_u \in H^1(Q)$. Hence, we can multiply equation \eqref{E3.4} by $z_v$ and perform an integration by parts. Using \eqref{E3.10}, \eqref{E2.16} and the fact that $\varphi_u(x,T) = z_v(x,0) = 0$ in $\Omega$, we get
\begin{align*}
&F'(u)v = \int_Q(y_u - y_d)z_v\dx\dt\\
& = \int_Q\big[-\frac{\partial\varphi_u}{\partial t}z_v + \nabla\varphi_u\nabla z_v + R'(y_u)\varphi_u z_v - (K^*[u]\varphi_u)z_v\big]\dx\dt\\
& = \int_Q\big[\frac{\partial z_v}{\partial t}\varphi_u + \nabla\varphi_u\nabla z_v + R'(y_u)z_v\varphi_u - (K[u]z_v)\varphi_u\big]\dx\dt\\
& = \int_Q\big(K[v]y_u + g_v\big)\varphi_u\dx\dt,
\end{align*}
which proves \eqref{E3.3}.
\endproof

We continue by studying the function $j:\MT \longrightarrow \mathbb{R}$. Since $j$ is Lipschitz and convex, we know that it has a nonempty subdifferential and possesses directional derivatives at every point $u \in \MT$ and in any direction $v \in \MT$. They will be denoted by $\partial j(u)$ and $j'(u;v)$, respectively.

Let us recall  some properties of $\partial j(u)$ and $j'(u;v)$; see \cite{CKK2016} and \cite{Casas-Kunisch2014} for similar results.

\begin{proposition}
If $\lambda \in \partial j(\bar u)$ with $\bar u \neq 0$ and $\lambda \in C[0,T]$, then the following properties hold
\begin{align}
& 
{\|\lambda\|_{\CT} = 1},\label{E3.11}\\
& \begin{array}{l}\text{supp}(\bar u^+) \subset \{t \in [0,T] : \lambda(t) = +1\},\\\text{supp}(\bar u^-) \subset \{t \in [0,T] : \lambda(t) = -1\},\end{array}\label{E3.12}
\end{align}
where $\bar u = \bar u^+ - \bar u^-$ is the Jordan decomposition of the measure 
{$\bar u$}.
\label{P3.2}
\end{proposition}

\begin{proof}
By definition of the 
{subdifferential}, we have
\begin{equation}
\langle u - \bar u,\lambda\rangle_{\MT,\CT} + j(\bar u) \le j(u) \ \ \forall u \in \MT.
\label{E3.13}
\end{equation}
Taking $u = 0$ and $u = 2\bar u$, respectively, in \eqref{E3.13} we deduce that $\langle \bar u,\lambda\rangle_{\MT,\CT} = j(\bar u)$. Hence \eqref{E3.13} implies that
\[
\langle u,\lambda\rangle_{\MT,\CT} \le  j(u) = \|u\|_\MT\ \ \forall u \in \MT.
\]
Now, for every $s \in [0,T]$ we take $u = \pm\delta_{s}$ in the above inequality. This leads to
\begin{equation}\label{R1E313}
|\lambda(s)| \le 1\ \ \forall s \in 
{[0,T]}.
\end{equation}
By the established properties, we find
{
\[
\|\bar u\|_\MT = j(\bar u) = \int_0^T\lambda(s)\dbus \le \int_0^T|\lambda(s)|\dbuas \le \int_0^T\dbuas = \|\bar u\|_\MT,
\]
therefore
\[
\int_0^T[1 - |\lambda(s)|]\dbuas = 0 \text{ and } \int_0^T\lambda(s)\dbus = \int_0^T|\lambda(s)|\dbuas.
\]
The 
{second identity and  \eqref{R1E313} imply} \eqref{E3.11}. Let us prove \eqref{E3.12}. From \eqref{E3.11} we infer
\begin{align*}
&0 \le \int_0^T(1 - \lambda(s))\, \textup{d}\bar u^+(s) + \int_0^T(1 + \lambda(s))\, \textup{d}\bar u^-(s)\\
& = \int_0^T \dbuas - \int_0^T \lambda(s)\, \dbus = \int_0^T(1 - |\lambda(s)|)\, \dbuas = 0.
\end{align*}
Hence, we get
\[
\int_0^T(1 - \lambda(s))\, \textup{d}\bar u^+(s) = \int_0^T(1 + \lambda(s))\, \textup{d}\bar u^-(s) = 0,
\]
which proves \eqref{E3.12}.}
\end{proof}

Now we study the directional derivatives of $j$. Following \cite{Casas-Kunisch2014}, we introduce another notation. Given $u, v \in \MT$, we consider the Lebesgue decomposition of $v$ with respect to $|u|$: $v = v_a + v_s$, where $v_a$ is the absolutely continuous part of $v$ with respect to $|u|$ and $v_s$ is the singular part; see, for instance, \cite[Chapter 6]{Rudin70}. We denote by $h_v$ the Radon-Nikodym derivative of $v_a$ with respect to $|u|$,  i.e. $dv_a = h_vd|u|$. Then we have
\begin{equation}
\|v\|_\MT = \|v_a\|_\MT + \|v_s\|_\MT = \int_0^T|h_v(s)|\duas + \|v_s\|_\MT.
\label{E3.14}
\end{equation}

Moreover, it is obvious that $u$ is absolutely continuous with respect to $|u|$. We have $du = h \, d|u|$, $du^+ = h^+ \, d|u|$, and $du^- = h^- \, d|u|$, where $|h(s)| = 1$ for every $s \in [0,T]$.

In the next statement, we derive the expression for the directional derivatives of $j$.

\begin{proposition}
For every $u, v \in \MT$, we have
\begin{equation}
j'(u;v) = \int_0^Th_v(s)\dus + \int_0^T\, \textup{d}|v_s|(s).
\label{E3.15}
\end{equation}
\label{P3.3}
\end{proposition}

We refer to \cite[Proposition 3.3]{Casas-Kunisch2014} for the proof.

\begin{theorem}
Let $\bar u$ be a local solution of \Pb. Then there exist $\bar y \in Y \cap C(\bar Q \cup \bar Q_-)$, $\bar\varphi \in H^1(Q) \cap C(\bar Q \cup \bar Q_+)$ and $\bar\lambda \in \CT \cap \partial j(\bar u)$ such that
\begin{align}
&\left\{\begin{array}{rcll}\displaystyle\frac{\partial\bar y}{\partial t} - \Delta\bar y + R(\bar y) &=
&\displaystyle K[\bar u]\bar y + g_{\bar u} \ &\text{ in } Q,\vspace{2mm}\\
\displaystyle \partial_n\bar y &=& 0 \ &\text{ on } \Sigma,\\[1ex]
\bar y(x,t) &=& y_0(x,t) \ &\text{ in }Q_-, \end{array}\right.\label{E3.16}\\
&\left\{\begin{array}{rcll}\displaystyle -\frac{\partial\bar\varphi}{\partial t} - \Delta\bar\varphi + R'(\bar y)\bar\varphi &=&\displaystyle K^*[\bar u]\bar\varphi + \bar y - y_d \ &\text{ in } Q,\vspace{2mm}\\ \displaystyle \partial_n\bar\varphi &=& 0 \ &\text{ on } \Sigma,\\[1ex]\bar\varphi &=& 0 \ &\text{ in } Q_+, \end{array}\right.\label{E3.17}\\
&\bar\lambda(s) = - \frac{1}{\nu}\int_{-s}^T\int_\Omega\bar y(x,t)\bar\varphi(x,t+s)\dx\dt\quad \forall s \in [0,T], \label{E3.18}
\end{align}
where $Q_+ = \Omega \times [T,2T]$.
\label{T3.3}
\end{theorem}

\begin{proof}
The existence and uniqueness of solutions to \eqref{E3.16} and \eqref{E3.17} have  already been discussed in Theorem \ref{T2.1} and Proposition \ref{P3.1}. Notice that the condition $\bar\varphi(x,T) = 0$ in $\Omega$ has been extended to $\bar\varphi(x,t) = 0$ in $Q_+ = \Omega \times [T,2T]$. It is obvious that this extension by $0$ defines a continuous function in $\bar Q \cup \bar Q_+$. Now, we define $\bar\lambda$ by \eqref{E3.18}. The continuity of $\bar y$ and $\bar\varphi$ implies that $\bar\lambda \in \CT$. It remains to prove that $\bar\lambda \in \partial j(\bar u)$. To this end, we use that $\bar u$ is a local minimizer of \Pb. Hence, for any $u \in \MT$, we get from the convexity of $j$ and \eqref{E3.3} that
\begin{align}
&0 \le \lim_{\rho \searrow 0}\frac{J(\bar u + \rho(u - \bar u)) - J(\bar u)}{\rho} \le F'(\bar u)(u - \bar u) + \nu j(u) - \nu j(\bar u)\notag\\
&= \int_Q\bar\varphi(K[u - \bar u]\bar y + g_{u - \bar u})\dx\dt + \nu j(u) - \nu j(\bar u).\label{E3.19}
\end{align}

\begin{align*}
&\int_Q\bar\varphi\, (K[u - \bar u]\bar y + g_{u - \bar u})\,\dx\,\dt\\
&= \int_\Omega\int_0^T\bar\varphi(x,t)\Big(\int_0^T\bar y(x,t-s)\, \textup{d}(u - \bar u)(s)\Big)\,\dt\,\dx\\
&= \int_0^T\int_\Omega\int_0^T\bar y(x,t-s)\bar\varphi(x,t)\dt\, \dx\, \textup{d}(u - \bar u)(s)\\
&= \int_0^T\int_\Omega\int_{-s}^{T-s}\bar y(x,\tau)\bar\varphi(x,\tau + s)\dtau\, \dx\, \textup{d}(u - \bar u)(s)\\
&=\int_0^T\left(\int_{-s}^T\int_\Omega\bar y(x,\tau)\bar\varphi(x,\tau+s)\dx\,\dtau\,\right) \textup{d}(u - \bar u)(s)\\
& = -\nu\int_0^T\bar\lambda(s) \, \textup{d}(u - \bar u)(s).
\end{align*}
Combining this with \eqref{E3.19}, we find
\[
\int_0^T\bar\lambda(s) \, \textup{d}(u - \bar u)(s) + j(\bar u) \le j(u)\quad \forall u \in \MT.
\]
This is the definition of $\bar\lambda \in \partial j(\bar u)$.
\end{proof}

From Proposition \ref{P3.2} and Theorem \ref{T3.3} we deduce the following sparsity structure of the optimal control $\bar u$.

\begin{corollary}
Let $\bar u$ be a local minimum of \Pb and let $\bar y$, $\bar\varphi$ and $\bar\lambda$ satisfy the optimality system \eqref{E3.16}-\eqref{E3.18}, then 
{if $\bar u \not\equiv 0$}
\begin{align}
&
{\ \|\bar\lambda\|_\CT = 1},\label{E3.20}\\
& \begin{array}{l}\text{supp}(\bar u^+) \subset \{t \in [0,T] : \bar\lambda(t) = +1\},\\\text{supp}(\bar u^-) \subset \{t \in [0,T] : \bar\lambda(t) = -1\},\end{array}\label{E3.21}
\end{align}
where $\bar u = \bar u^+ - \bar u^-$ is the Jordan decomposition of the measure $\bar u$.
\label{C3.1}
\end{corollary}

\begin{proposition}
There exists $\bar\nu > 0$ such that $0$ is the only solution of \Pb 
{for every $\nu \ge \bar\nu$}.
\label{P3.4}
\end{proposition}

\begin{proof}
Let $\bar u$ be a solution of \Pb. From the inequality $J(\bar u) \le J(0)$ we deduce that
\begin{equation}
\|\bar y\|_{L^2(Q)} \le C_1 < \infty\ \text{ and }\ \|\bar u\|_\MT \le \frac{C_2}{\nu}
\label{E3.22}
\end{equation}
for some constants independent of $\nu$. Arguing similarly as in the proof of inequality \eqref{E2.13}, we get from equation \eqref{E3.17} that
\begin{equation}
\|\bar\varphi\|_{C([0,T],L^2(\Omega))} \le C_3\|\bar y - y_d\|_{L^2(Q)} \le C_4.
\label{E3.23}
\end{equation}
According to \eqref{E3.22}, $C_3$ and $C_4$ are independent of $\nu \ge 1$. Now, from \eqref{E3.18}, \eqref{E3.22} and \eqref{E3.23} we get
\begin{align*}
|\bar\lambda(s)| &\le \frac{1}{\nu}\int_{-s}^T\|\bar y(t)\|_{L^2(\Omega)}\|\bar\varphi(t+s)\|_{L^2(\Omega)}\, \dt\\
&\le \frac{1}{\nu}\Big(\int_{-T}^T\|\bar y(t)\|_{L^2(\Omega)}\, \dt\Big)\|\bar\varphi\|_{C([0,T],L^2(\Omega))}\\
& = \frac{1}{\nu}\Big(\int_{-T}^0\|y_0(t)\|_{L^2(\Omega)}\, \dt + \int_{0}^T\|\bar y(t)\|_{L^2(\Omega)}\, \dt\Big)\|\bar\varphi\|_{C([0,T],L^2(\Omega))}\\
& \le \frac{\sqrt{T}}{\nu}\big(\|y_0\|_{L^2(Q_-)} + \|\bar y\|_{L^2(Q)}\big)\|\bar\varphi\|_{C([0,T],L^2(\Omega))} \le \frac{\sqrt{T}}{\nu}\big(\|y_0\|_{L^2(Q_-)} + C_1\big)C_4.
\end{align*}
If we take $\bar\nu > \max\big\{1,\sqrt{T}\big(\|y_0\|_{L^2(Q_-)} + C_1\big)C_4\big\}$ we infer that $|\bar\lambda(s)| < 1$ $\forall s \in [0,T]$. Then, \eqref{E3.21} implies that $\bar u \equiv 0$.
\end{proof}
\section{Discretization of the Control Space}
\label{S5}
\setcounter{equation}{0}

In this section  we are going to consider the approximation of  $\MT$ by finite dimensional subspaces $\U$. Associated to each space $\U$ we define a new problem \Pbt. Then, we analyze the convergence of the solutions of \Pbt . First we consider a grid of points $0 = t_0 < t_1 < \ldots < t_{N_\tau} = T$. We set $\tau_k = t_k-t_{k-1}$ for $1 \le k \le N_\tau$ and $\tau = \max_{1 \le k \le N_\tau}\tau_k$. 
{We also set $I_k = (t_{k - 1},t_k]$ for $1 \le k \le N_\tau$, and $I_0 = \{0\}$.} Associated with this grid we define the space
\[
\U = \{u_\tau = \sum_{k = 0}^{N_\tau}u_k\delta_{t_k} : (u_k)_{k = 0}^{N_\tau} \in \mathbb{R}^{N_\tau + 1}\},
\]
where $\delta_{t_k}$ denotes the Dirac measure centered at $t_k$.
Thus, $\U$ has dimension $N_\tau + 1$ and $\U$ is a vector subspace of $\MT$. Now, we introduce the linear mapping
\[
\Lambda_\tau:\MT \longrightarrow \U \ \text{ defined by }\ \Lambda_\tau u = \sum_{k = 0}^{N_\tau}u(I_k) \delta_{t_k}.
\]
The following proposition states some properties of this mapping.
\begin{proposition}
The following statements hold
\begin{enumerate}
\item $\|\Lambda_\tau u\|_\MT \le \|u\|_\MT$ $\forall u \in \MT$.
\item $\Lambda_\tau u \stackrel{*}{\rightharpoonup} u$ in $\MT$ $\forall u \in \MT$.
\item $\lim_{\tau \to 0}\|\Lambda_\tau u\|_\MT = \|u\|_\MT$ $\forall u \in \MT$.
\end{enumerate}
\label{P4.1}
\end{proposition}

\begin{proof}
\textit{1. -} It is obtained as follows
\[
\|\Lambda_\tau u\|_\MT = \sum_{k = 0}^{N_\tau}|u(I_k)| \le \sum_{k = 0}^{N_\tau}|u|(I_k) = |u|([0,T]) = \|u\|_\MT.
\]

\textit{2. -} Let us take $y \in C[0,T]$. Given an arbitrary $\varepsilon > 0$, the continuity of $y$ implies that there exists $\tau_\varepsilon > 0$ such that
\begin{equation}
|y(t) - y(s)| < \varepsilon \ \ \forall s, t \in [0,T] \text{ such that } |t - s| < \tau_\varepsilon.
\label{E4.1}
\end{equation}
Then for every $\tau < \tau_\varepsilon$ we have
\begin{align*}
&|\langle u - \Lambda_\tau u,y\rangle| = \Big|\int_0^T y(s) \dus - \sum_{k = 0}^{N_\tau}y(t_k)u(I_k)\Big| = \Big|\sum_{k = 0}^{N_\tau}\int_{I_k}[y(s) - y(t_k)]\dus\Big|\\
&\le \sum_{k = 0}^{N_\tau}\int_{I_k}|y(s) - y(t_k)|\duas \le \varepsilon\|u\|_\MT.
\end{align*}
Since $y$ is an arbitrary element of $C[0,T]$, this proves that $\Lambda_\tau u \stackrel{*}{\rightharpoonup} u$ in $\MT$.

\textit{3. -} Combining \textit{2} and \textit{1} we get
\[
\|u\|_\MT \le \liminf_{\tau \to 0}\|\Lambda_\tau u\|_\MT \le \limsup_{\tau \to 0}\|\Lambda_\tau u\|_\MT \le\|u\|_\MT,
\]
which concludes the proof.
\end{proof}

Now, for every $\tau > 0$ we consider the control problem with discretized controls
\[
\Pbt \min_{u_\tau \in \U}  J(u_\tau) = \frac{1}{2}\int_Q|y_{u_\tau} - y_d|^2\dx\dt + \nu\sum_{k = 0}^{N_\tau}|u_k|,
\]
From Lemma \ref{L3.1} we deduce the continuity of the functional $J:\U \longrightarrow \mathbb{R}$. Therefore, taking into account that $\U$ is a finite dimensional vector space and $J$ is coercive on $\U$, we deduce the existence of at least one global solution $\bar u_\tau$ of \Pbt. Let us study the sparse structure of the solutions $\bar u_\tau$ of \Pbt. We denote by $j_\tau:\U \longrightarrow \mathbb{R}$ the restriction of $j$ to $\U$:
\[
j_\tau(u_\tau) = j(u_\tau) = \sum_{k= 0}^{N_\tau}|u_k|.
\]
We identify the dual of $\U$ with $\mathbb{R}^{N_\tau + 1}$ as follows:
\[
\forall \lambda_\tau = (\lambda_k)_{
{k = 0}}^{N_\tau} \in \mathbb{R}^{
{N_\tau + 1}} \text{ and } \forall u_\tau = \sum_{k = 0}^{N_\tau}u_k\delta_{t_k} \text{ we set } \langle\lambda_\tau,u_\tau\rangle = \sum_{k = 0}^{N_\tau}\lambda_ku_k.
\]
Then Proposition \ref{P3.2} is reformulated as follows.
\begin{proposition}
With the above notation, we have $\lambda_\tau \in \partial j_\tau(\bar u_\tau)$ if and only if the following identity holds
\begin{equation}
\lambda_k \left\{\begin{array}{ll} = +1 & \text{if } \bar u_k > 0,\\ = -1 & \text{if } \bar u_k < 0,\\ \in [-1,+1] & \text{if } \bar u_k = 0,\end{array}\right.\ \ 0 \le k \le N_\tau.
\label{E4.2}
\end{equation}
\label{P4.2}
\end{proposition}
\begin{proof}
By definition of the subdifferential we have that $\lambda_\tau \in \partial j_\tau(\bar u_\tau)$ if and only if
\[
\sum_{k = 0}^{N_\tau}\lambda_k(u_k - \bar u_k) +\sum_{k=0}^{N_\tau}|\bar u_k| \le \sum_{k=0}^{N_\tau}|u_k|\quad \forall u_\tau \in \U.
\]
The above relation is equivalent to
\[
\lambda_k(u_k - \bar u_k) + |\bar u_k| \le |u_k|\quad \forall \, 0 \le k \le N_\tau \text{ and } \forall u_\tau \in \U.
\]
Obviously, the above inequalities are equivalent to \ref{E4.2}.
\end{proof}

Using this proposition, the following theorem can be proved as Theorem \ref{T3.3}.
\begin{theorem}
Let $\bar u_\tau$ be a local solution of \Pbt. Then there exist $\bar y_\tau \in Y \cap C(\bar Q \cup \bar Q_-)$, $\bar\varphi_\tau \in H^1(Q) \cap C(\bar Q \cup \bar Q_+)$ and $\bar\lambda_\tau \in \partial j_\tau(\bar u_\tau)$ such that
\begin{align}
&\left\{\begin{array}{rcll}\displaystyle\frac{\partial\bar y_\tau}{\partial t} - \Delta\bar y_\tau + R(\bar y_\tau) &=
&\displaystyle K[\bar u_\tau]\bar y_\tau + g_{\bar u_\tau} \ &\text{ in } Q,\vspace{2mm}\\
\displaystyle \partial_n\bar y_\tau &=& 0 \ &\text{ on } \Sigma,\\[1ex]
\bar y_\tau(x,t) &=& y_0(x,t) \ &\text{ in }Q_-, \end{array}\right.\label{E4.3}\\
&\left\{\begin{array}{rcll}\displaystyle -\frac{\partial\bar\varphi_\tau}{\partial t} - \Delta\bar\varphi_\tau + R'(\bar y_\tau)\bar\varphi_\tau &=&\displaystyle K^*[\bar u_\tau]\bar\varphi_\tau + \bar y_\tau - y_d \ &\text{ in } Q,\vspace{2mm}\\ \displaystyle \partial_n\bar\varphi_\tau &=& 0 \ &\text{ on } \Sigma,\\[1ex]\bar\varphi_\tau &=& 0 \ &\text{ in } Q_+, \end{array}\right.\label{E4.4}\\
&\bar\lambda_k = - \frac{1}{\nu}\int_{-t_k}^T\int_\Omega\bar y_\tau(x,t)\bar\varphi_\tau(x,t+t_k)\dx\dt\quad \forall\, 0 \le k \le N_\tau. \label{E4.5}
\end{align}
\label{T4.1}
\end{theorem}

Combining Proposition \ref{P4.2} and \eqref{E4.5} we deduce the following corollary.
\begin{corollary}
Let $\bar u_\tau$ be a local minimum of 
{\Pbt} with $\bar u_\tau\not\equiv 0$ and let $\bar y_\tau$, $\bar\varphi_\tau$ and $\bar\lambda_\tau$ satisfy \eqref{E4.3}-\eqref{E4.5}, then
\begin{align}
&\max_{0\leq k\leq N_{\tau}} |\bar\lambda_k| = 1,\label{E4.6}\\
&\left\{ \begin{array}{l}\bar u_k > 0  \Rightarrow \bar\lambda_k = +1,\\\bar u_k < 0 \Rightarrow \bar\lambda_k = -1.\end{array}\right.\label{E4.7}
\end{align}
\label{C4.1}
\end{corollary}

Finally, we analyze the convergence of the above discretization.

\begin{theorem}
Let $\{\bar u_{\tau}\}_{\tau \searrow 0}$ be a sequence of discrete controls such that every control $\bar u_{\tau}$ is a solution  of \Pbt. This sequence is bounded in $\MT$. Any weak$^*$ limit point of a subsequence is a solution of \Pb, and $J(\bar u_\tau) \to \inf\Pb$ as $\tau \searrow 0$. In addition, if $\tau_j \to 0$ and $u_{\tau_j} \stackrel{*}{\rightharpoonup}\bar u$ in $\MT$, then $\lim_{j \to \infty}\|\bar u_{\tau_j}\|_\MT = \|\bar u\|_\MT$ and $\lim_{j \to \infty}\|\bar y_{\tau_j} - \bar y\|_Y = 0$, where $\bar y_{\tau_j}$ and $\bar y$ denote the states associated to $\bar u_{\tau_j}$ and $\bar u$, respectively.
\label{T4.2}
\end{theorem}

\begin{proof}
The boundedness in $\MT$ is an immediate consequence of the inequalities $J(\bar u_\tau) \le J(0)$ $\forall \tau$. Assume that $u_{\tau_j} \stackrel{*}{\rightharpoonup}\bar u$ in $\MT$  as $j \to \infty$. From Lemma \ref{L3.1} we have $\bar y_{\tau_j} \to \bar y$ in $Y$. Let $\tilde u$ be a solution of \Pb. Then, we get with Proposition \ref{P4.1} and Lemma \ref{L3.1}
\[
J(\bar u) \le \liminf_{j \to \infty}J(\bar u_{\tau_j}) \le \limsup_{j \to \infty}J(\bar u_{\tau_j})\le \limsup_{j \to \infty}J(\Lambda_{\tau_j}\tilde u) = J(\tilde u) = \inf\Pb.
\]
From these inequalities we deduce that $\bar u$ is a solution of \Pb. Consequently, we have that $J(\bar u) = J(\tilde u)$. Then, using again the above inequalities we get that $J(\bar u_{\tau_j}) \to J(\bar u)$. This convergence along with $y_{\bar u_{\tau_j}} \to \bar y$ in $L^2(Q)$ implies that $\|\bar u_{\tau_j}\|_\MT \to \|\bar u\|_\MT$.
\end{proof}
\section{Numerical examples}\setcounter{equation}{0}
Next we present some test examples to illustrate our results.
To solve the problem, we have used a Tikhonov regularization of $\Pbt$.
 For $c>0$, we consider the problem
\[\Pbtc \min_{u_\tau\in\U} J(u_\tau)+\frac{1}{2c} \sum_{k=0}^{N_\tau}u_k^2.\]
{The} first order optimality conditions for $\Pbtc$ read {as follows}:
\begin{theorem} Let $u_{\tau}^c$ be a local solution of $\Pbtc$. Then there exist $y_{\tau}^c\in Y\cap C(\bar Q\cup\bar Q_{-})$, $\varphi_{\tau}^c\in H^1(Q)\cap C(\bar Q\cup\bar Q_{+})$ and $\lambda_{\tau}^c\in \partial j_{\tau}(u_\tau^c)$ such that
\begin{align}
&\left\{\begin{array}{rcll}\displaystyle\frac{\partial y_\tau^c}{\partial t} - \Delta y_\tau^c + R( y_\tau^c) &=
&\displaystyle K[ u_\tau^c] y_\tau^c + g_{u_\tau^c} \ &\text{ in } Q,\vspace{2mm}\\
\displaystyle \partial_n y_\tau^c &=& 0 \ &\text{ on } \Sigma,\\[1ex]
 y_\tau^c(x,t) &=& y_0(x,t) \ &\text{ in }Q_-, \end{array}\right.\label{E5.1}\\
&\left\{\begin{array}{rcll}\displaystyle -\frac{\partial\varphi_\tau^c}{\partial t} - \Delta\varphi_\tau^c + R'(y_\tau^c)\varphi_\tau^c &=&\displaystyle K^*[ u_\tau^c]\varphi_\tau^c +  y_\tau^c - y_d \ &\text{ in } Q,\vspace{2mm}\\ \displaystyle \partial_n\varphi_\tau^c &=& 0 \ &\text{ on } \Sigma,\\[1ex]\varphi_\tau^c &=& 0 \ &\text{ in } Q_+, \end{array}\right.\label{E5.2}\\
&-\nu\lambda_k^c =  \int_{-t_k}^T\int_\Omega y_\tau^c(x,t)\varphi_\tau^c(x,t+t_k)\dx\dt+\frac{1}{c}u_k^c\quad \forall\, 0 \le k \le N_\tau. \label{E5.3}
\end{align}
\label{T5.1}
\end{theorem}
Taking into account Proposition \ref{P4.2}, the condition on the subgradient $\lambda_\tau^c\in\partial j_\tau(u_\tau^c)$ can be written as
\begin{equation}
 u_k^c = \max\{0, u_k^c+C(\lambda_k^c-1)\}+\min\{0, u_k+C( \lambda_k^c+1)\}\ \forall C>0.\label{E5.4}
\end{equation}
We solve the system \eqref{E5.1}--\eqref{E5.4} by a semi-smooth Newton method. The linear system arising at each iteration is reduced to a linear system for the active part of the control variable that is solved using GMRES. To solve the delay linear parabolic partial differential equations that appear in the process we consider the standard continuous piecewise linear finite elements in space and piecewise constant discontinuous Galerkin method in time, i.e.,  dG(0)cG(1) discretization. The problem is solved first for some small value  $c=c_0>0$ with initial guess $u\equiv 0$. Given the solution for some value $c=c_k$, $k\geq 0$, it is taken as the initial guess to solve the optimality system for $c=c_{k+1}>c_k$. The process stops when further changes of $c_k$ do not {alter} the solution in a significant way. We choose $C=c\nu$ at every iteration.

In all our examples, the reaction term is of the form \[R(y) = \rho(y-y_1)(y-y_2)(y-y_3).\]

\setcounter{theorem}{0}

\begin{example}[Example with known critical point]\label{Example52}\end{example}
To test the discretization and the optimization algorithm, we first {construct} an example with known solution of the optimality system given in Theorem \ref{T3.3}.

Consider $\Omega=(0,1)\subset\mathbb{R}$, $T=1$, $\rho =1/3$, $-y_1=y_3=\sqrt{3}$, $y_2=0$. Define
\[\bar u = u^{*}\delta_{t^*},\]
where $t^*=0.5$ and $u^*=-7.7$. With this control, we compute (an approximation of) its related state $\bar y$ solving the state equation with prehistory \[y_0(x,t) = \frac{1+t}{5}\sin^2(\pi x).\]
{For} this example, we use a discretization of 257 evenly spaced nodes both in space and time to solve the parabolic partial differential equations.

Next we define
{
\[
\bar\varphi(x,t) = \left\{\begin{array}{cl}
\cos^2\left(\frac{\pi}{2}t\right)&\mbox{ if }0\leq t \leq T,\vspace{2mm}\\
0 & \mbox{ if } T\leq t.
\end{array}\right.
\]}
This function satisfies the boundary and final conditions of the adjoint state equation. {Moreover,}  we define
\[y_d(x,t) = \partial_t \bar\varphi(x,t) -R'(\bar y)\bar\varphi +\bar y +\int_0^T\bar\varphi(x,t+s)\mathrm{d}\bar u(s).\]
{Taking into account that $-\Delta \bar\varphi =0$,} we have that $\bar\varphi$ satisfies the adjoint state equation{; see} Figure \ref{F51} for a picture of the computed target.

Finally, we compute
\[\nu \bar\lambda(s)=-\int_0^T\int_\Omega \bar y(x,t-s)\bar\varphi(x,t) \dx \dt.\]
With our choices of $t^*$, $u^*$ {,} and $\bar\varphi(x,t)$, we have that ${s \mapsto} \nu\bar\lambda(s)$ is a strictly convex function in $[0,1]$ that has a minimum at $t^*$ such that $\nu\bar\lambda(t^*)=-3.39817\times 10^{-4}$ (see Figure \ref{F51}). If we define $\nu = 3.39817\times 10^{-4}$, we have that $|\bar\lambda(s)| < 1$ for all $s\neq t^*$ and $\bar\lambda(t^*)=-1$,
and therefore  ($\bar u$, $\bar y$, $\bar\varphi$, $\bar\lambda$) satisfies {the} first order optimality conditions for problem $\Pb$ with data $\nu$ and $y_d$.

\begin{figure}
  \centering
  \includegraphics[width=.45\textwidth]{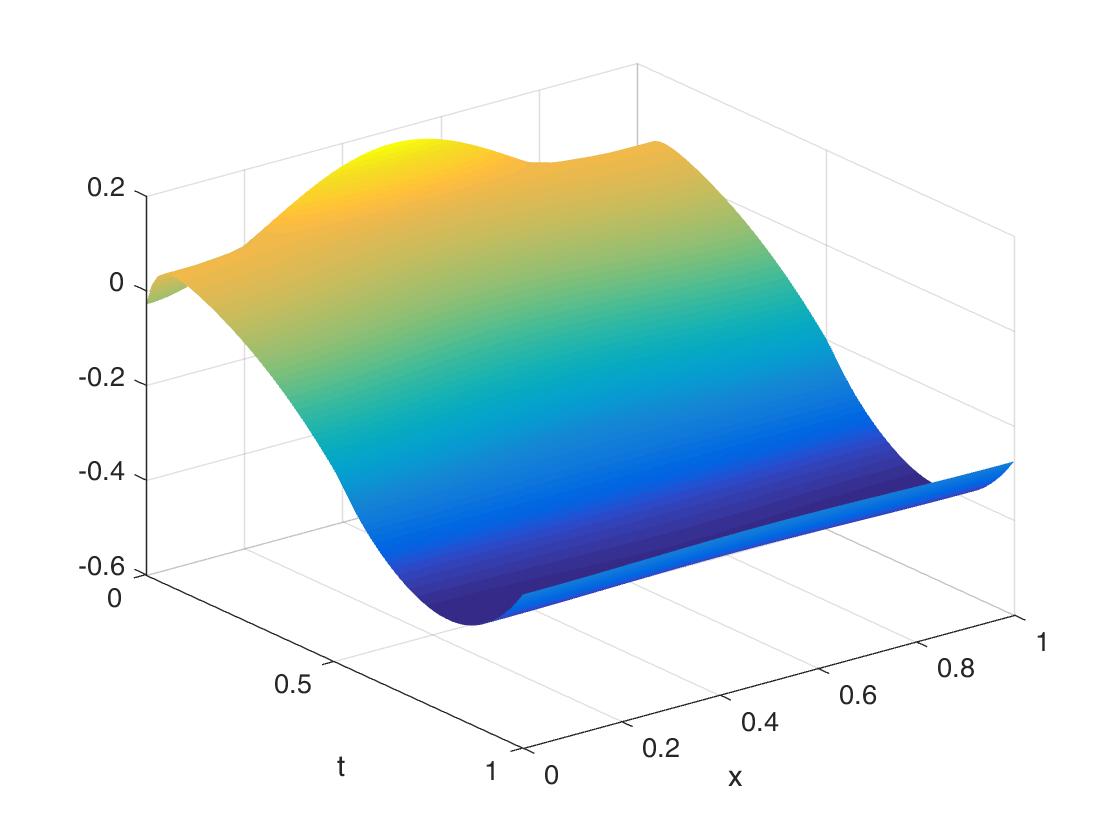}
  \includegraphics[width=.45\textwidth]{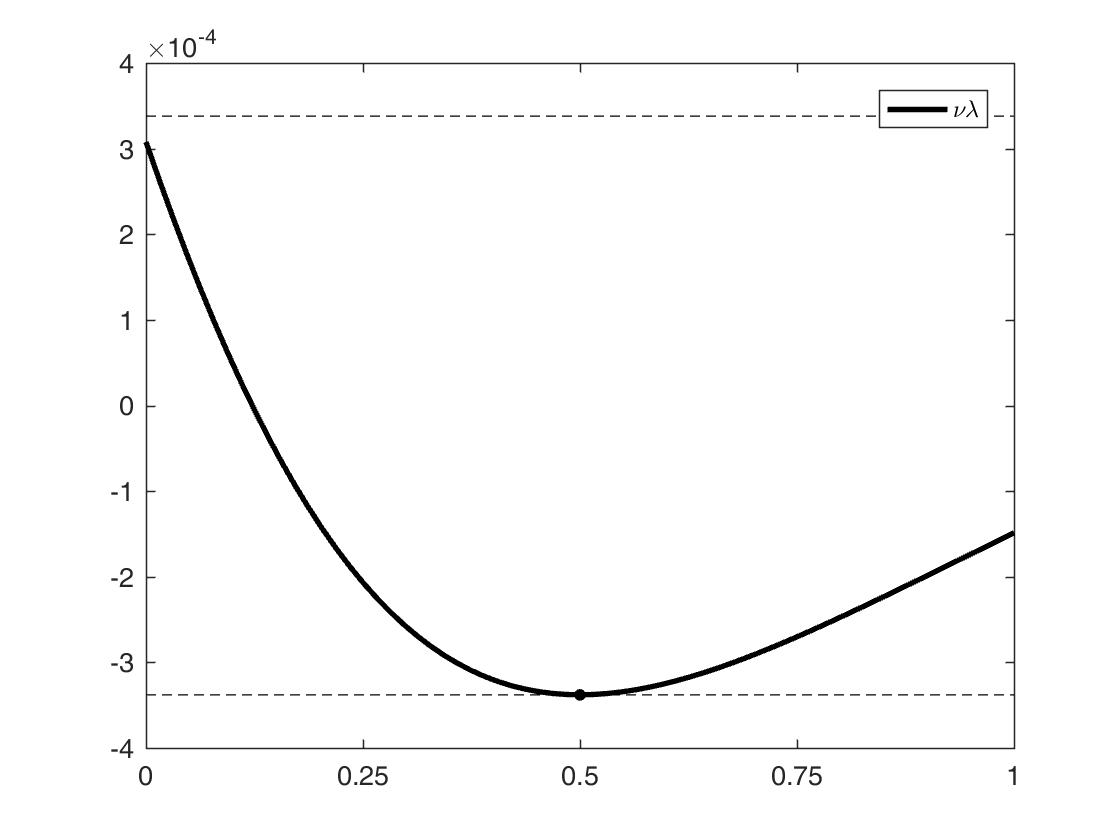}
  \caption{Target (left) and $\nu\bar\lambda(t)$ for Example \ref{Example52}}\label{F51}
\end{figure}

The values of the differentiable and non-differentiable parts of the functional are
\[F(\bar u) =  9.067\times 10^{-5}\mbox{ and }\nu    j(\bar u) =  2.617\times 10^{-3}.\]

If we solve the problem with a discretization of the control space such that $\bar{u}\in\U$, we recover the original solution with {four} digits of accuracy.
To set a more realistic scenario, we test our software in grids with constant time step $\tau=3^{-k}$, $k=2:5$, so that $\bar{u}\not\in\U$.

{The numerical results are displayed in Table \ref{TABLA51}.}  Notice that we are able to confirm all the results {of} Theorem \ref{T4.2}. We write $t^{-}_\tau=t^*-\tau/2$ and $t^+_{\tau}=t^*+\tau/2$ for the closest points to $t^*$ in the control mesh.  The values corresponding to the original solution $\bar u$ {are} included in the last row of the table.

\begin{table}[h!]
\[\begin{array}{c|c|c|c|c|c}
\tau   & \bar u_\tau                              & \|\bar y-\bar y_\tau\|_{Y} &\|\bar u_\tau\|_{\MT} & F(\bar u_\tau) & \nu j(\bar u_\tau) \\ \hline
3^{-2} & -3.33\delta_{t^{-}_\tau} -4.50\delta_{t^{+}_\tau} & 3.15e-5                   & 7.831                & 9.132e{-5}     & 2.661e{-3}   \\
3^{-3} & -3.30\delta_{t^{-}_\tau}-4.43\delta_{t^{+}_\tau}   &  9.40e-7                     & 7.730                & 9.107e{-5}     & 2.627e{-3}  \\
3^{-4} & -2.55\delta_{t^{-}_\tau}-5.17\delta_{t^{+}_\tau}   &  2.87e-8                     & 7.720                & 9.067e{-5}     & 2.623e{-3}  \\
3^{-5} & -0.11\delta_{t^{-}_\tau}-7.61\delta_{t^{+}_\tau}& 9.72e-9                     & 7.719                & 9.066e-{5}     & 2.623e{-3}\\ \hline
\mbox{exact}      &  -7.7\delta_{t^*}                        & 0                            & 7.7                  & 9.067e{-5}     & 2.617e{-3}
\end{array}\]
\caption{Results of Example with known solution\label{TABLA51}}
\end{table}

\begin{example}[{Sensitivity to the regularization parameter $\nu$}]\end{example}
For the same problem as {above}, we illustrate how the solution changes as $\nu$ varies. {It} can be expected {that} the value of $F$ decreases as $\nu$ decreases, and both $\|\bar u_\tau\|_{\MT}$ as well as the number of points in the support of $\bar u_\tau$ increase. As {it} was proved in Proposition \ref{P3.4}, there is a 
{$\bar\nu>0$} such that 
{the optimal control is zero for $\nu \ge \bar\nu$}. In this example, we use a discretization of 65 {equidistant nodes} both in space and time. We use the same time grid for the control discretization. Our results are shown in Table \ref{Tabla52}{.}

\begin{table}[h!]
\[\begin{array}{cccc}
    \nu & F_\tau (\bar u_\tau) & \|\bar u_\tau\|_{\MT} & \sharp\mathrm{supp}\bar u_\tau \\ \hline
    1e-1 & 8.85e-2 & 0 & 0 \\
    6e-2 & 8.85e-2 & 0 & 0 \\
    5e-2 & 8.74e-2 & 0.02 & 1 \\
    1e-2 & 1.95e-2 & 3.54 & 1 \\
    1e-3 & 3.79e-3 & 7.36 & 2 \\
    1e-4 & 5.97e-5 & 7.95 & 2 \\
    1e-5 & 4.31e-5 & 8.83 & 4 \\
    1e-6 & 3.38e-5 & 13.3 & 8 \\
    1e-7 & 2.12e-5 & 59.4 & 31 \\
    1e-8 & 1.85e-5 & 170.0 & 56
  \end{array}
  \]
  \caption{Sensitivity of the norm and the support of the optimal control\label{Tabla52}}
\end{table}

\begin{example}[Recovering the solution of a {system with} non-local Pyragas {feedback control}]\end{example}
We consider the data of Example 1 in \cite{nestler_schoell_troeltzsch2015}, {namely}  $\Omega=(-20,20)$, $T=40$, $\rho = 1$, $y_1=0$, $y_2=0.25$, $y_3=1$. The prehistory is given by
\[y_0(x,t) = \frac12(y_1+y_3)+\frac{1}{2}(y_1-y_3)\tanh\left(\frac{(y_3-y_1)}{2\sqrt{2}}(x-ct)\right),\]
where
\[c=\frac{y_1+y_3-2y_2}{\sqrt{2}}.\]
The desired state is the solution of the state equation with delay term given by the measure $u_d\in\MT$ defined as
\[\int_{[0,T]} y(x,t-s)du_d(s) = \kappa\left(\frac{1}{t_b-t_a}\int_{t_a}^{t_b} y(x,t-s) ds -y(x,t)\right),\]
with parameters $\kappa = 0.5$, $t_a=0.456$, $t_b=0.541$. We fix the parameter $\nu=1e-2$ --this is big enough to obtain a combination of {Dirac measures} -- and {we} will look for solutions in $\mathcal{M}[0,1]$. Since for the given delay term $\|u_d\|_{\mathcal{M}[0,1]} = 2\kappa = 1$ and $F(u_d)=0$,  it holds $J(u_d)=\nu$.
Numerically, we obtain the solution
\[\bar u = -0.240821\delta_{s=0}+0.246667\delta_{s=1}{.}
\]
{The associated values of the objective are} $F(\bar u) = 5.3e-5$ and $\nu j(\bar u) = 4.87e-3$,  {hence} $J(\bar u) = 0.49\nu$. The target and the optimal state are illustrated in Figure \ref{F_NSTEx1}. To solve the equations, we have used a space mesh with 513 evenly spaced nodes and a time grid with a variable stepsize and 778 nodes. The discretization of the control has been done with a grid of 101 equidistant nodes in $[0,1]$.

\begin{figure}
  \centering
  \includegraphics[width=.45\textwidth]{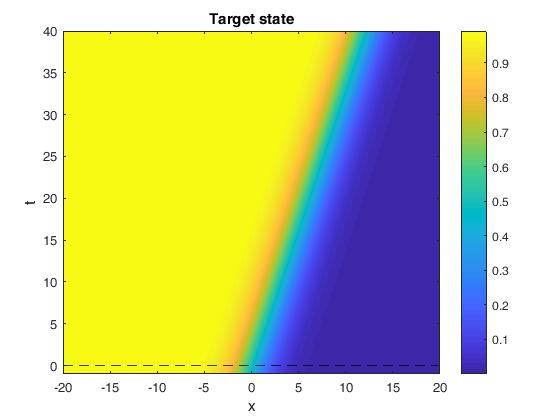}
  \includegraphics[width=.45\textwidth]{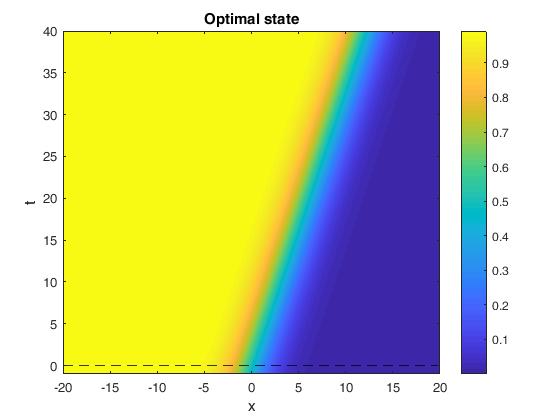}
  \caption{Target (left) obtained with a non-local Pyragas feedback control and optimal state (right)}\label{F_NSTEx1}
\end{figure}

\begin{example}[{Steering} the system to an unstable equilibrium point]\end{example} {Here, } our data are $\Omega=(0,1)$, $T=2$, $\rho = 1$, $y_1=0$, $y_2=0.25$, $y_3=1$. The prehistory is given by $y_0(x,t)\equiv y_3$, which {is} a stable equilibrium point and the target is $y_d(x,t)\equiv y_2$, which is an unstable equilibrium point. Since the data do not depend on $x$ and the boundary conditions are satisfied, the problem is equivalent to controlling a nonlinear delay ODE. We fix $\nu=1e-3$ and consider the tracking only on $[T/2,T]$. Therefore, here we redefine the differentiable part of the functional $J(u)$ by
\begin{equation}F(u) = \frac12\int_{T/2}^T\int_\Omega(y_u(x,t)-y_d(x,t))dxdt.\label{E5.5}\end{equation}
With $N_\tau = 512$ time steps, we obtain
\[\bar u = -1.304\delta_{0.418} +0.134\delta_{1.977}+0.220\delta_{1.978},\]
$F(\bar u) = 1.29e-4${,} and $J(\bar u) = 1.787e-3$.

\begin{example}[Changing the period of an incoming wave]\label{Example5}\end{example} {We use the} same data as in the previous example{, but} $y_0(x,t) = \cos^2(2\pi t)/2$ and $y_d(x,t)=\cos^2(\pi t)/2$. We fix $\nu=1e-3$ and {take} $F(u)$ as defined in \eqref{E5.5}.
{By} a discretization with $N_\tau=256$ time steps, we obtain
\[\bar u = 0.3188\delta_{0} -1.5499 \delta_{0.4219} -0.9964 \delta_{0.8047} + 2.7233\delta_{1.5234} \]
{and the objective values} $F(\bar u) = 6.57e-4$ and $J(\bar u) = 6.25e-3$.  Prehistory, target, uncontrolled state, and the state associated with the computed optimal delay control are illustrated in Figure \ref{F55}, where we plot the functions for $x=0$.
\begin{figure}[h!]
  \centering
  \includegraphics[width=0.5\textwidth]{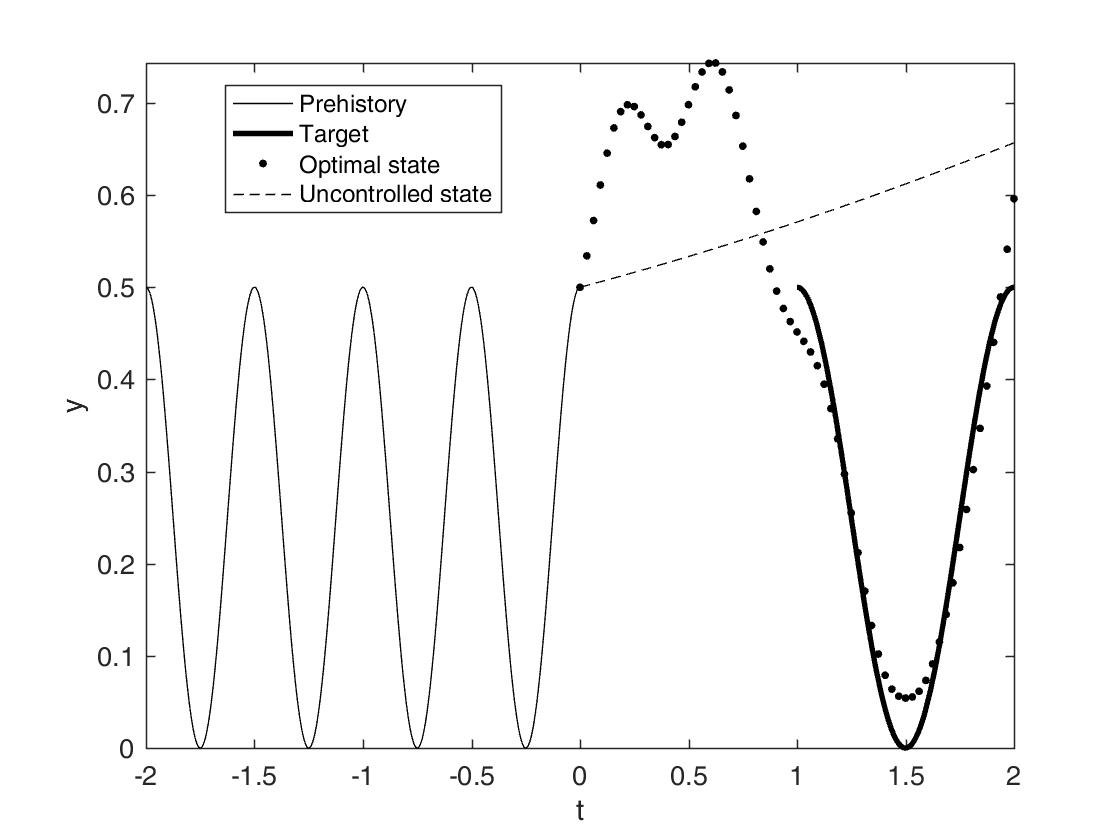}
  \caption{Data and solution for Example \ref{Example5}}\label{F55}
\end{figure}

%

\begin{thebibliography}{10}

\bibitem{Ahmed2003}
N.~U. Ahmed.
\newblock Existence of optimal controls for a general class of impulsive
  systems on {B}anach spaces.
\newblock {\em SIAM J. Control Optim.}, 42(2):669--685, 2003.
\newblock URL: \url{http://dx.doi.org/10.1137/S0363012901391299}, \href
  {http://dx.doi.org/10.1137/S0363012901391299}
  {\path{doi:10.1137/S0363012901391299}}.

\bibitem{Banks_Burns1978}
H.~T. Banks and J.~A. Burns.
\newblock Hereditary control problems: numerical methods based on averaging
  approximations.
\newblock {\em SIAM J. Control Optimization}, 16(2):169--208, 1978.
\newblock URL: \url{http://dx.doi.org/10.1137/0316013}, \href
  {http://dx.doi.org/10.1137/0316013} {\path{doi:10.1137/0316013}}.

\bibitem{Banks_Burns_Cliff1981}
H.~T. Banks, J.~A. Burns, and E.~M. Cliff.
\newblock Parameter estimation and identification for systems with delays.
\newblock {\em SIAM J. Control Optim.}, 19(6):791--828, 1981.
\newblock URL: \url{http://dx.doi.org/10.1137/0319051}, \href
  {http://dx.doi.org/10.1137/0319051} {\path{doi:10.1137/0319051}}.

\bibitem{Banks_Manitius1974}
H.~T. Banks and Andrzej Manitius.
\newblock Application of abstract variational theory to hereditary systems--a
  survey.
\newblock {\em IEEE Trans. Automatic Control}, AC-19:524--533, 1974.

\bibitem{Bainov_Mishev1991}
D.~D. Ba\u\i~nov and D.~P. Mishev.
\newblock {\em Oscillation theory for neutral differential equations with
  delay}.
\newblock Adam Hilger, Ltd., Bristol, 1991.

\bibitem{Bellen_Zennaro2003}
Alfredo Bellen and Marino Zennaro.
\newblock {\em Numerical methods for delay differential equations}.
\newblock Numerical Mathematics and Scientific Computation. The Clarendon
  Press, Oxford University Press, New York, 2003.
\newblock URL:
  \url{http://dx.doi.org/10.1093/acprof:oso/9780198506546.001.0001}, \href
  {http://dx.doi.org/10.1093/acprof:oso/9780198506546.001.0001}
  {\path{doi:10.1093/acprof:oso/9780198506546.001.0001}}.

\bibitem{Bellman1954}
Richard Bellman.
\newblock {\em A survey of the mathematical theory of time-lag, retarded
  control, and hereditary processes}.
\newblock The Rand Corporation, Santa Monica, Calif., 1954.
\newblock With the assistance of John M. Danskin, Jr.

\bibitem{Casas97}
E.~Casas.
\newblock {P}ontryagin's principle for state-constrained boundary control
  problems of semilinear parabolic equations.
\newblock {\em SIAM J. Control Optim.}, 35(4):1297--1327, 1997.

\bibitem{CKK2016}
E.~Casas, F.~Kruse, and K.~Kunisch.
\newblock Optimal control of semilinear parabolic equations by {BV}-functions.
\newblock {\em SIAM J. Control Optim.}, to appear, 2016.

\bibitem{Casas-Kunisch2014}
E.~Casas and K.~Kunisch.
\newblock Optimal control of semilinear elliptic equations in measure spaces.
\newblock {\em SIAM J. Control Optim.}, 52(1):339--364, 2013.

\bibitem{casas_ryll_troeltzsch2014}
E.~Casas, C.~Ryll, and F.~Tr{\"o}ltzsch.
\newblock Sparse optimal control of the {S}chl\"ogl and {F}itz{H}ugh-{N}agumo
  systems.
\newblock {\em Computational Methods in Applied Mathematics}, 13:415--442,
  2014.
\newblock \href {http://dx.doi.org/10.1515/cmam-2013-0016}
  {\path{doi:10.1515/cmam-2013-0016}}.

\bibitem{Colonius_Hinrichsen1976}
F.~Colonius and D.~Hinrichsen.
\newblock Optimal control of hereditary differential systems.
\newblock In {\em Recent theoretical developments in control ({P}roc. {C}onf.,
  {U}niv. {L}eicester, {L}eicester, 1976)}, pages 215--239. Academic Press,
  London-New York, 1978.
\newblock With discussion.

\bibitem{Erneux2009}
T.~Erneux.
\newblock {\em Applied delay differential equations}, volume~3 of {\em Surveys
  and Tutorials in the Applied Mathematical Sciences}.
\newblock Springer, New York, 2009.

\bibitem{Grisvard85}
P.~Grisvard.
\newblock {\em Elliptic Problems in Nonsmooth Domains}.
\newblock Pitman, Boston-London-Melbourne, 1985.

\bibitem{Hale_Verduyn1993}
Jack~K. Hale and Sjoerd~M. Verduyn~Lunel.
\newblock {\em Introduction to functional-differential equations}, volume~99 of
  {\em Applied Mathematical Sciences}.
\newblock Springer-Verlag, New York, 1993.
\newblock URL: \url{http://dx.doi.org/10.1007/978-1-4612-4342-7}, \href
  {http://dx.doi.org/10.1007/978-1-4612-4342-7}
  {\path{doi:10.1007/978-1-4612-4342-7}}.

\bibitem{Jeong_Hwang2015}
Jin-Mun Jeong and Hae-Jun Hwang.
\newblock Optimal control problems for semilinear retarded functional
  differential equations.
\newblock {\em J. Optim. Theory Appl.}, 167(1):49--67, 2015.
\newblock URL: \url{http://dx.doi.org/10.1007/s10957-015-0726-8}, \href
  {http://dx.doi.org/10.1007/s10957-015-0726-8}
  {\path{doi:10.1007/s10957-015-0726-8}}.

\bibitem{Kappel_Kunisch1981}
F.~Kappel and K.~Kunisch.
\newblock Spline approximations for neutral functional-differential equations.
\newblock {\em SIAM J. Numer. Anal.}, 18(6):1058--1080, 1981.
\newblock URL: \url{http://dx.doi.org/10.1137/0718072}, \href
  {http://dx.doi.org/10.1137/0718072} {\path{doi:10.1137/0718072}}.

\bibitem{Kunisch1981}
Karl Kunisch.
\newblock Approximation schemes for nonlinear neutral optimal control systems.
\newblock {\em J. Math. Anal. Appl.}, 82(1):112--143, 1981.
\newblock URL: \url{http://dx.doi.org/10.1016/0022-247X(81)90228-6}, \href
  {http://dx.doi.org/10.1016/0022-247X(81)90228-6}
  {\path{doi:10.1016/0022-247X(81)90228-6}}.

\bibitem{kyrychko_blyuss_schoell11}
Y.~N. Kyrychko, K.~B. Blyuss, and E.~Sch{\"o}ll.
\newblock Amplitude death in systems of coupled oscillators with
  distributed-delay coupling.
\newblock {\em Eur. Physi. J. B}, 84:307--315, 2011.
\newblock \href {http://dx.doi.org/10.1140/epjb/e2011-20677-8}
  {\path{doi:10.1140/epjb/e2011-20677-8}}.

\bibitem{Lad-Sol-Ura68}
O.A. Ladyzhenskaya, V.A. Solonnikov, and N.N. Ural'tseva.
\newblock {\em Linear and Quasilinear Equations of Parabolic Type}.
\newblock American Mathematical Society, 1988.

\bibitem{Lober2014}
J.~L\"ober, R.~Coles, J.~Siebert, H.~Engel, and E.~Sch\"oll.
\newblock Control of chemical wave propagation.
\newblock {\em arXiv}, 1403:3363, 2014.

\bibitem{Mordukhovich_Wang2009}
Boris~S. Mordukhovich, Dong Wang, and Lianwen Wang.
\newblock Optimal control of delay-differential inclusions with functional
  endpoint constraints in infinite dimensions.
\newblock {\em Nonlinear Anal.}, 71(12):e2740--e2749, 2009.
\newblock URL: \url{http://dx.doi.org/10.1016/j.na.2009.06.022}, \href
  {http://dx.doi.org/10.1016/j.na.2009.06.022}
  {\path{doi:10.1016/j.na.2009.06.022}}.

\bibitem{Mordukhovich_Wang2010}
Boris~S. Mordukhovich, Dong Wang, and Lianwen Wang.
\newblock Optimization of delay-differential inclusions in infinite dimensions.
\newblock {\em Pac. J. Optim.}, 6(2):353--374, 2010.

\bibitem{nestler_schoell_troeltzsch2015}
P.~Nestler, E.~Sch\"oll, and F.~Tr\"oltzsch.
\newblock Optimization of nonlocal time-delayed feedback controllers.
\newblock {\em Computational Optimization and Applications}, DOI
  10.1007/s10589-015-9809-6, published online 2015.

\bibitem{pyragas1992}
K.~Pyragas.
\newblock Continuous control of chaos by self-controlling feedback.
\newblock {\em Phys. Rev. Lett.}, A 170:421, 1992.

\bibitem{pyragas2006}
K.~Pyragas.
\newblock Delayed feedback control of chaos.
\newblock {\em Phil. Trans. R. Soc}, A 364:2309, 2006.

\bibitem{Richard2003}
Jean-Pierre Richard.
\newblock Time-delay systems: an overview of some recent advances and open
  problems.
\newblock {\em Automatica J. IFAC}, 39(10):1667--1694, 2003.
\newblock URL: \url{http://dx.doi.org/10.1016/S0005-1098(03)00167-5}, \href
  {http://dx.doi.org/10.1016/S0005-1098(03)00167-5}
  {\path{doi:10.1016/S0005-1098(03)00167-5}}.

\bibitem{Rudin70}
W.~Rudin.
\newblock {\em Real and Complex Analysis}.
\newblock McGraw-Hill, London, 1970.

\bibitem{schoell_schuster2008}
E.~Sch\"oll and H.G. Schuster.
\newblock {\em Handbook of Chaos Control}.
\newblock Wiley-VCH, Weinheim, 2008.

\bibitem{Showalter1997}
R.~E. Showalter.
\newblock {\em Monotone operators in {B}anach space and nonlinear partial
  differential equations}, volume~49 of {\em Mathematical Surveys and
  Monographs}.
\newblock American Mathematical Society, Providence, RI, 1997.

\bibitem{siebert_alonso_baer_schoell14}
J.~Siebert, S.~Alonso, M.~B{\"a}r, and E.~Sch{\"o}ll.
\newblock Dynamics of reaction-diffusion patterns controlled by asymmetric
  nonlocal coupling as a limiting case of differential advection.
\newblock {\em Physical Review E}, 89, 052909, 2014.
\newblock \href {http://dx.doi.org/10.1103/PhysRevE.89.052909}
  {\path{doi:10.1103/PhysRevE.89.052909}}.

\bibitem{siebert_schoell14}
J.~Siebert and E.~Sch{\"o}ll.
\newblock Front and turing patterns induced by mexican-hat-like nonlocal
  feedback.
\newblock {\em Europhys. Lett.}, 109, 40014, 2015.

\bibitem{Troltzsch2010}
F.~{Tr\"{o}ltzsch}.
\newblock {\em Optimal Control of Partial Differential Equations: Theory,
  Methods and Applications}, volume 112 of {\em Graduate Studies in
  Mathematics}.
\newblock American Mathematical Society, Philadelphia, 2010.

\bibitem{unger_vonwolfersdorf1995}
F.~Unger and L.~von Wolfersdorf.
\newblock On a control problem for memory kernels in heat conduction.
\newblock {\em Z. Angew. Math. Mech.}, 75(5):365--370, 1995.
\newblock URL: \url{http://dx.doi.org/10.1002/zamm.19950750505}, \href
  {http://dx.doi.org/10.1002/zamm.19950750505}
  {\path{doi:10.1002/zamm.19950750505}}.

\bibitem{vonwolfersdorf1993}
L.~von Wolfersdorf.
\newblock On optimality conditions in some control problems for memory kernels
  in viscoelasticity.
\newblock {\em Z. Anal. Anwendungen}, 12(4):745--750, 1993.

\end{thebibliography}
\end{document}